\documentclass[11pt]{amsart}

\usepackage[colorlinks=true, pdfstartview=FitV, linkcolor=blue, citecolor=blue, urlcolor=blue, breaklinks=true]{hyperref}

\usepackage{amsmath,amsfonts,amssymb,amsthm,epsfig,comment}
\usepackage[all]{xy}
\usepackage[usenames]{color}  

\newtheorem{Theorem}{Theorem}[section]
\newtheorem{Proposition}[Theorem]{Proposition}
\newtheorem{Lemma}[Theorem]{Lemma}
\newtheorem{Corollary}[Theorem]{Corollary}

\theoremstyle{definition}
\newtheorem{Definition}[Theorem]{Definition}

\newtheorem{Remark}[Theorem]{Remark}
\newtheorem{Example}[Theorem]{Example}
\newtheorem{PropDef}[Theorem]{Proposition/Definition}
\newtheorem{Error}[Theorem]{Error}

\newcommand{\arxiv}[1]{\href{http://arxiv.org/abs/#1}{\tt arXiv:\nolinkurl{#1}}}

\newcommand{\g}{\mathfrak{g}}

\newcommand{\GL}{\mathrm{GL}}

\DeclareMathOperator{\Hom}{Hom}
\DeclareMathOperator{\End}{End}
\DeclareMathOperator{\Rep}{Rep}

\newcommand{\Z}{\mathbb{Z}}
\newcommand{\st}{\mathrm{st}}
\newcommand{\N}{\mathbb{N}}

\newcommand{\C}{\mathbb{C}}
\newcommand{\Gr}{\mathrm{Gr}}
\newcommand{\NGr}{\mathrm{NGr}}
\DeclareMathOperator{\im}{im}

\DeclareMathOperator{\Id}{Id}

\DeclareMathOperator{\socle}{socle}
\DeclareMathOperator{\Aut}{Aut}
\DeclareMathOperator{\cha}{char}
\DeclareMathOperator{\Ext}{Ext}
\newcommand{\Q}{\mathbb{Q}}

\newcommand{\cP}{\mathcal{P}}

\newcommand{\id}{\mathrm{id}}

\numberwithin{equation}{section}

\begin{document}

\title{Quiver grassmannians, quiver varieties and the preprojective algebra}

\author{Alistair Savage}
\address{A.~Savage: University of Ottawa, Ottawa, ON, Canada}
\email{alistair.savage@uottawa.ca}
\author{Peter Tingley}
\address{P.~Tingley: Massachusetts Institute of Technology, Cambridge, MA, USA}
\email{ptingley@math.mit.edu}

\thanks{The research of the first author was supported by the Natural
Sciences and Engineering Research Council of Canada.  The second
author was supported by Australian Research Council grant DP0879951
and NSF grant DMS-0902649.}

\begin{abstract}
Quivers play an important role in the representation theory of
algebras, with a key ingredient being the path algebra and the
preprojective algebra. Quiver grassmannians are varieties of
submodules of a fixed module of the path or preprojective algebra.
In the current paper, we study these objects in detail. We show that
the quiver grassmannians corresponding to nilpotent submodules of certain
injective modules are homeomorphic to the lagrangian quiver
varieties of Nakajima which have been well studied in the context of
geometric representation theory.  We then refine this result by
finding quiver grassmannians which are homeomorphic to the Demazure
quiver varieties introduced by the first author, and others which
are homeomorphic to the graded/cyclic quiver varieties defined by
Nakajima.  The Demazure quiver grassmannians allow us to describe
injective objects in the category of locally nilpotent modules of
the preprojective algebra. We conclude by relating our construction
to a similar one of Lusztig using projectives in place of
injectives.  In an appendix added after the first version of the
current paper was released, we show how subsequent results of
Shipman imply that the above homeomorphisms are in fact isomorphisms
of algebraic varieties.
\end{abstract}

\maketitle \thispagestyle{empty}
\setcounter{tocdepth}{1} 
\vspace{-1.01cm}
\tableofcontents

\vspace{-1cm}

\textbf{Note:} This version of the paper incorporates an erratum to the published version.  See Appendix~\ref{sec:erratum} for details.


\section*{Introduction}

Quivers play a fundamental role in the theory of associative
algebras and their representations. Gabriel's theorem, which states
a precise relationship between indecomposable representations of
certain quivers and root systems of associated Lie algebras,
indicated that the representation theory of quivers was also
intimately connected to the representation theory of Kac-Moody
algebras.  This eventually lead to the Ringel-Hall construction of
quantum groups and the quiver variety constructions of Lusztig and
Nakajima.

Fix a quiver (directed graph) $Q=(Q_0,Q_1)$ with vertex set $Q_0$
and arrow set $Q_1$.  The corresponding path algebra $\C Q$ is the
algebra spanned by the set of directed paths, with multiplication
given by concatenation. There is a natural grading $\C Q =
\bigoplus_n (\C Q)_n$ of the path algebra by length of paths.
Representations of a quiver are equivalent to representations (or
modules) of its path algebra.  Note that $(\C Q)_0$-modules are
simply $Q_0$-graded vector spaces, and in particular all $\C
Q$-modules are $Q_0$-graded. For a $\C Q$-module $V$ and $u \in \N
Q_0$, the associated \emph{quiver grassmannian} is the set
$\Gr_Q(u,V)$ of all $\C Q$-submodules  of $V$ of graded dimension
$u$. The {\it nilpotent quiver grassmannian} is the subset $\NGr_Q(u,V)$ of $\Gr_Q(u,V)$ consisting of nilpotent modules. These natural objects (or closely related ones) can be found in
several places in the literature. For instance, they appear in
\cite{Cra96,Sch92} in the study of spaces of morphisms of $\C
Q$-modules and in \cite{CC06,CK06,DWZ09} in connection with the
theory of cluster algebras.  Geometric properties have been studied
in \cite{CR08,Sza09,Wol09} and representation theoretic properties
in \cite{Fed10,GLS06,Lus98,Lus00b,Nak03b,Rei08}.

Let $\g$ be the Kac-Moody algebra whose Dynkin diagram is the
underlying graph of $Q$ (the graph obtained by forgetting the
orientation of all arrows) and let $\tilde Q$ be the \emph{double
quiver} obtained from $Q$ by adding an oppositely oriented arrow
$\bar a$ for every $a \in Q_1$. One is often interested in modules
of the preprojective algebra $\cP=\cP(Q)$, which is a certain
natural quotient of the path algebra $\C \tilde Q$ and inherits the
grading.  In particular, $\cP$-modules are also $\C \tilde
Q$-modules. To each vertex $i \in Q_0$, we have an associated
one-dimensional simple $\cP$-module $s^i$.  For $w  = \sum_i w_i i
\in \N Q_0$, we let $s^w = \bigoplus_i (s^i)^{\oplus w_i}$ be the
corresponding semisimple module. By Baer's Theorem,
the category of $\cP$-modules has enough injectives, so we can
define $q^w$ to be the injective hull of $s^w$. One of the main
results of the current paper is that the nilpotent quiver grassmannian
$\NGr_{\tilde Q}(v,q^w)$ is homeomorphic to the lagrangian Nakajima
quiver variety $\mathfrak{L}(v,w)$ used to give a geometric
realization of irreducible highest weight representations of $\g$
(see \cite{Nak94,Nak98}).  In addition, for each $\sigma$ in the
Weyl group of $\g$, there is a natural finite-dimensional submodule
$q^{w,\sigma}$ of $q^w$ such that the quiver grassmannian
$\Gr_{\tilde Q}(v,q^{w,\sigma})$ is homeomorphic to the Demazure
quiver variety $\mathfrak{L}_\sigma(v,w)$ defined by the first
author \cite{Sav04a}.  Since Nakajima's realization of highest
weight representations and the first author's realization of
Demazure modules depend only on the topological information of the
spaces involved, such homeomorphisms allow one to replace quiver
varieties by quiver grassmannians in the constructions.  This change
of setting affords some advantages.  In particular, it avoids the
description as a moduli space.  One can view it as a uniform way of
picking a representative from each orbit in the original moduli
space descriptions.

Quiver grassmannians admit natural group actions.  We describe these
actions and show that certain special cases agree, under the
homeomorphisms described above, with well-studied groups actions on
Nakajima quiver varieties.  In this way, we are able to give a
quiver grassmannian realization of the cyclic/graded quiver
varieties used by Nakajima to define $t$-analogs of $q$-characters
of quantum affine algebras \cite{Nak04}.

The injective modules $q^w$ are locally nilpotent if and only if the
quiver $Q$ is of finite or affine type.  However,
the submodules $q^{w, \sigma}$ are always nilpotent.  The limit
$\tilde q^w$ of these submodules is the injective hull of the
semisimple module $s^w$ in the category of locally nilpotent
$\cP$-modules, giving us a description of the indecomposable
injectives in this category. Furthermore, $\NGr(v, q^w)$ is naturally identified with $\Gr(u, \tilde q^w)$, so we also obtain that this ordinary quiver variety is homeomorphic to (and in fact isomorphic to, using the results in the appendix) $\mathfrak{L}(v,w)$.

Lusztig has previously presented a canonical bijection between the
points of the lagrangian Nakajima quiver variety and the points of a
type of quiver grassmannian inside a projective (as opposed to
injective) object.  In finite type, the projective objects are also
injective. It turns out that, on the level of geometric realizations
of representations of finite type $\g$, the two constructions are
related by the Chevalley involution.  Outside of finite type, there
are some other subtle yet important differences between the two
constructions.  In particular, through the use of the distinguished
modules $q^{w,\sigma}$ mentioned above, one can always consider
quiver grassmannians of submodules of a fixed
\emph{finite-dimensional} module of the preprojective algebra. Thus,
one can avoid working with infinite-dimensional objects.

Motivated by an earlier version of the current paper \cite{ST09},
I.~Shipman~\cite{Shi10} has recently proven that the canonical
bijection given by Lusztig and mentioned above is, in fact, an
isomorphism of algebraic varieties.  We have added an appendix
explaining how this result allows us to conclude that the maps
between quiver grassmannians and lagrangian Nakajima quiver
varieties described in the current paper are also isomorphisms of
algebraic varieties.

Throughout the current paper, we work over the field $\C$ of complex
numbers.  While many results hold in more generality, this
assumption will streamline the exposition and several results we
quote in the literature are stated over $\C$. We will always use the
Zariski topology and do not assume that algebraic varieties are
irreducible. We let $\N = \Z_{\ge 0}$ and denote the fundamental
weights and simple roots of a Kac-Moody algebra by $\omega_i$ and
$\alpha_i$ respectively.

This paper is organized as follows.  In
Section~\ref{sec:quivers-overview} we review some results on
quivers, path algebras and preprojective algebras.  In
Section~\ref{sec:modules} we discuss various module categories of
these objects and introduce our main object of study, the quiver
grassmannian.  We review the definition of the quiver varieties of
Lusztig and Nakajima in Section~\ref{sec:QVs} and realize these as
quiver grassmannians in Section~\ref{sec:QVs-and-QGs}.  In
Section~\ref{sec:actions-and-graded-versions} we introduce a natural
group action and show how it can be used to recover group actions
typically constructed on quiver varieties.  We also define
graded/cyclic versions of quiver grassmannians. In
Section~\ref{sec:geom-reps} we use quiver grassmannians to give a
geometric realization of integrable highest weight representations
of a symmetric Kac-Moody algebra and discuss the compatibility of
this construction with the natural nesting of quiver grassmannians.
Finally, in Section~\ref{sec:injective-versus-projective} we discuss
a precise relationship between our construction and a similar one
due to Lusztig.  Appendix~\ref{Shipman-Appendix}, added after the
appearance of \cite{Shi10}, provides a proof that the maps between
quiver grassmannians and quiver varieties described in the current
paper are isomorphisms of algebraic varieties.


\section{Quivers, path algebras, and preprojective algebras}
\label{sec:quivers-overview}

In this section we briefly review the relevant definitions
concerning quivers.  We refer the reader to
\cite{DDPW08,Rin98,Sav06a} for further details.

A \emph{quiver} is a directed graph.  That is, it is a quadruple $Q
= (Q_0,Q_1,s,t)$ where $Q_0$ and $Q_1$ are sets and $s$ and $t$ are
maps from $Q_1$ to $Q_0$.  We call $Q_0$ and $Q_1$ the sets of
\emph{vertices} and \emph{directed edges} (or \emph{arrows})
respectively. For an arrow $a \in Q_1$, we call $s(a)$ the
\emph{source} of $a$ and $t(a)$ the \emph{target} of $a$.  Usually
we will write $Q=(Q_0,Q_1)$, leaving the maps $s$ and $t$ implied.
The quiver $Q$ is said to be \emph{finite} if $Q_0$ and $Q_1$ are
finite.  A \emph{loop} is an arrow $a$ with $s(a)=t(a)$. In this
paper, all quivers will be assumed to be finite and without loops. A
quiver is said to be of \emph{finite type} if the underlying graph
of $Q$ (i.e the graph obtained from $Q$ by forgetting the orientation
of the edges) is a Dynkin diagram of finite $ADE$ type. Similarly,
it is of \emph{affine} (or \emph{tame}) \emph{type} if the
underlying graph is a Dynkin diagram of affine type and of
\emph{indefinite} (or \emph{wild}) \emph{type} if the underlying
graph is a Dynkin diagram of indefinite type.

A \emph{path} in $Q$ is a sequence $\beta = a_l a_{l-1} \cdots a_1$
of arrows such that $t(a_i) = s(a_{i+1})$ for $1 \le i \le l-1$.  We
call $l$ the \emph{length} of the path. We let $s(\beta) = s(a_1)$
and $t(\beta) = t(a_l)$ denote the initial and final vertices of the
path $\beta$.  For each vertex $i \in I$, we have a trivial path
$e_i$ with $s(e_i) = t(e_i)=i$.

The \emph{path algebra} $\C Q$ associated to a quiver $Q$ is the
$\C$-algebra whose underlying vector space has basis the set of
paths in $Q$, and with the product of paths given by concatenation.
More precisely, if $\beta=a_l \cdots a_1$ and $\beta' = b_m \cdots
b_1$ are two paths in $Q$, then $\beta \beta'=a_l \cdots a_1 b_m
\cdots b_1$ if $t(\beta')=s(\beta)$ and $\beta \beta'=0$ otherwise.
This multiplication is associative. There is a natural grading $\C Q
= \bigoplus_{n \ge 0} (\C Q)_n$ where $(\C Q)_n$ is the span of the
paths of length $n$.

Given a quiver $Q=(Q_0,Q_1)$, we define the \emph{double quiver}
associated to $Q$ to be the quiver $\tilde Q=(Q_0, \tilde Q_1)$
where
\[
    \tilde Q_1 = \bigcup_{a \in Q_1} \{a, \bar a\},\quad \text{where} \quad
    s(\bar a)=t(a),\ t(\bar a)=s(a).
\]
We then have a natural involution $\tilde Q_1 \to \tilde Q_1$ given
by $a \mapsto \bar a$ (where ${\bar {\bar a}}=a$).  The algebra
\[
  \mathcal{P}= \mathcal{P}(Q) = \C \tilde Q/ \sum_{a \in Q_1}
  (a \bar a - \bar a a)
\]
is called the \emph{preprojective algebra} associated to $Q$.  It
inherits a grading $\mathcal{P} = \bigoplus_{n \ge 0} \mathcal{P}_n$
from the grading on $\C Q$.  Up to isomorphism, the preprojective
algebra $\cP(Q)$ depends only on the underlying graph of $Q$. See
\cite[\S 12.15]{L91} for details.


\section{Modules of the path algebra and quiver grassmannians}
\label{sec:modules}

\subsection{Module categories}

For an associative algebra $A$, let $A$-Mod denote the category of
$A$-modules and $A$-mod the category of finite-dimensional
$A$-modules.  We will use the notation $V \in A$-Mod (resp. $V \in
A$-mod) to indicate that $V$ is an object in the category $A$-Mod
(resp. $A$-mod). Note that $\cP_0$-mod is equivalent to the category
of finite-dimensional $Q_0$-graded vector spaces whose morphisms are
linear maps preserving the grading, and we will often blur the
distinction between these two categories. Up to isomorphism, the
objects of $\cP_0$-mod are classified by their graded dimension.  We
denote the graded dimension of a module $V$ by $\dim_{Q_0} V =
\sum_i (\dim V_i) i \in \N Q_0$ and let $\dim_\C V = \sum_{i \in
Q_0} \dim V_i \in \N$.  We will sometimes view the graded dimension
$\dim_{Q_0} V$ of $V$ as its isomorphism class.

For $V, W \in \cP_0$-mod, we denote the set of $\cP_0$-module
morphisms from $V$ to $W$ by $\Hom_{\cP_0} (V,W)$. Under the
equivalence of categories above, $\Hom_{\cP_0}(V,W)$ is identified
with $\bigoplus_{i \in Q_0} \Hom_\C (V_i, W_i)$. We define
$\End_{\cP_0} V$ to be $\Hom_{\cP_0} (V,V)$ and $\GL_V =
\prod_{i \in Q_0} GL(V_i)$ to be group of invertible elements of
$\End_{\cP_0} V$. For $V \in \cP_0$-mod, we will write $U \subseteq
V$ to mean that $U$ is a $\cP_0$-submodule of $V$.  This is the
same as a $Q_0$-graded subspace.  Note that any $\cP$-module becomes
a $\cP_0$-module by restriction, and thus can be thought of as a
$Q_0$-graded vector space.

Suppose $A = \bigoplus_{n \ge 0} A_n$ is a graded algebra and $V$ is
an $A$-module.  Then $V$ is \emph{nilpotent} if there exists an $n
\in \N$ such that $A_k \cdot V=0$ for all $k \ge n$.  We say $V$ is \emph{locally
nilpotent} if for all $v \in V$, there exists $n \in \N$ such that
$A_k \cdot v = 0$ for all $k \ge n$.  We denote by $A$-lnMod the category of locally
nilpotent $A$-modules.  For $n \ge 0$, we define $A_{\ge n} =
\bigoplus_{k \ge n} A_k$ and we let $A_+ = A_{\ge 1}$.

\begin{Proposition} \label{prop:path-algebra-type}
For a quiver $Q$, the following are equivalent:
\begin{enumerate}
  \item \label{prop-item:PQ-fd} $\cP(Q)$ is finite-dimensional,
  \item \label{prop-item:PQ-nilpotent} all finite-dimensional $\cP(Q)$-modules are nilpotent,
  \item \label{prop-item:PQ-locally-nilpotent} all finite-dimensional
  $\cP(Q)$-modules are locally nilpotent, and
  \item \label{prop-item:Q-dynkin} $Q$ is of finite type.
\end{enumerate}
\end{Proposition}

\begin{proof}
The equivalence of \eqref{prop-item:PQ-fd} and
\eqref{prop-item:Q-dynkin} is well-known (see for example
\cite{Rei97}).  That \eqref{prop-item:PQ-nilpotent} implies
\eqref{prop-item:Q-dynkin} was proven by Crawley-Boevey \cite{Cra01}
and the converse was proven by Lusztig \cite[Proposition~14.2]{L91}.
Since a finite-dimensional module is nilpotent if and only if it is
locally nilpotent, \eqref{prop-item:PQ-nilpotent} is equivalent to
\eqref{prop-item:PQ-locally-nilpotent}.
\end{proof}

\subsection{Simple objects}

For each $i \in Q_0$, let $s^i$ be the simple $\C \tilde Q$-module
given by $s^i_i=\C$ and $s^i_j=0$ for $i \ne j$.  Then $s^i$ is also
naturally a $\mathcal{P}$-module which we also denote by $s^i$.

\begin{Lemma} \label{lem:simples}
  The set $\{s^i\}_{i \in Q_0}$ is a set of representatives of the isomorphism classes of
  simple objects of $\C \tilde Q$-lnMod and $\cP$-lnMod.  In particular, if
  $Q$ is of finite type, then $\{s^i\}_{i \in Q_0}$
  is a set of representatives of the isomorphism classes of simple objects of $\C \tilde Q$-mod
  and $\cP$-mod.
\end{Lemma}

\begin{proof}
  Any nonzero element of a simple locally nilpotent module $M$
  generates a finite-dimensional module which must be all of $M$.
  Therefore $M$ is finite-dimensional and hence nilpotent.
  Then $(\C \tilde Q)_+$ and $\mathcal{P}_+$ are two-sided ideals
  of $\C \tilde Q$ and
  $\mathcal{P}$ respectively that act nilpotently on any nilpotent
  module.  Therefore, simple nilpotent $\C \tilde Q$-modules and
  $\mathcal{P}$-modules are the same as simple $\C \tilde Q/(\C \tilde Q)_+$-modules
  and $\mathcal{P}/\mathcal{P}_+$-modules respectively.  Since
  \[
    \C \tilde Q/(\C \tilde Q)_+ \cong \mathcal{P}/\mathcal{P}_+ \cong \bigoplus_{i \in
    I} \C e_i,
  \]
  the first statement follows.  The second statement then follows
  from Proposition~\ref{prop:path-algebra-type}.
\end{proof}

\begin{Lemma} \label{lem:socles}
  Fix a quiver $Q$ and let $A$ be either $\C \tilde Q$ or $\cP(Q)$.  If $V \in
  A$-lnMod, then the socle of $V$ is $\{v \in V\ |\ A_+ \cdot v=0\}$.
\end{Lemma}

\begin{proof}
  It is clear that $\{v \in V\ |\ A_+ \cdot v=0\}$ is a sum of simple
  subrepresentations of $V$ and is thus contained in the socle
  of $V$.  Similarly, by Lemma~\ref{lem:simples}, any simple
  subrepresentation of $(V,x)$ is contained in $\{v \in V\ |\ A_+ \cdot v=0\}$.
\end{proof}

\subsection{Projective covers}

Recall that if $A$ is an associative algebra and $V$ is an
$A$-module, then a \emph{projective cover} of $V$ is a pair $(P,f)$
such that $P$ is a projective $A$-module and $f : P \to V$ is a
superfluous epimorphism of $A$-modules.  This means that $f(P)=V$
and $f(P') \ne V$ for all proper submodules $P'$ of $P$. We often
omit the homomorphism $f$ and simply call $P$ a projective cover of
$V$.

\begin{Definition}
  For $i \in Q_0$, let $p^i = \mathcal{P} e_i$.
\end{Definition}

\begin{Lemma} \label{lem:projectives}
  Assume $Q$ is a quiver of finite type.
  For $i \in Q_0$, $\{p^i\}_{i \in Q_0}$ is a set of representatives
  of the isomorphism classes of indecomposable projective $\mathcal{P}$-modules.
  Furthermore, $p^i$ is a projective cover of $s^i$.
\end{Lemma}

\begin{proof}
This follows from \cite[Prop.~4.8]{ARS97}.
\end{proof}

\begin{Lemma}
  Assume $Q$ is a quiver of affine (tame) or indefinite (wild) type.
  Then there exist $i \in Q_0$ for which the simple module $s^i$ does
  not have a projective cover.
\end{Lemma}

\begin{proof}
  Since the module $s^i$ is obviously cyclic, by
  \cite[Lemma~27.3]{AF92} it has a projective cover if and only if
  $s^i \cong \cP e/Ie$ for some idempotent $e \in \cP$ and some left
  ideal $I$ contained in the Jacobson radical of $\cP$.  Assume this
  is true for some idempotent $e$ and ideal $I$.  Then we must have
  $e=e_i$ and then $I$ would have to contain $\cP_+ e_i$, the ideal
  consisting of all paths of length at least one starting at vertex
  $i$.  We identify $\Z Q_0$ with the root lattice
  via $\sum v_j j \leftrightarrow \sum v_j \alpha_j$.
  Let $\beta$ be a minimal positive imaginary root and let $i$ be in
  the support of $\beta$ (i.e.\ $\beta = \sum \beta_j \alpha_j$ with
  $\beta_i>0$).
  By \cite[Theorem~1.2]{Cra01}, there is a simple
  module $T$ of $\cP$ whose dimension vector is $\beta$ and so,
  in particular, $\dim T_i \ne 0$.
  Since the simple module $T$ cannot be killed by $\cP_+ e_i$ (since
  then $T_i$ would be a proper submodule), $\cP_+ e_i$ is
  not contained in the Jacobson radical of $\cP$.  This contradicts
  the fact that $I$ is contained in the Jacobson radical.
\end{proof}

\subsection{Injective hulls}

Recall that if $A$ is an associative algebra and $V$ is an
$A$-module, then an \emph{injective hull} of $V$ is an injective
$A$-module $E$ that is an essential extension of $V$ (that is, $V$
is a submodule of $E$ and any nonzero submodule of $E$ intersects
$V$ nontrivially).  By the Baer's Theorem \cite{Bae40}, the category
$\cP$-Mod has enough injectives.  In particular, the simple modules
$s^i$ have injective hulls. Here we give an explicit description of
these injective hulls in the finite type case, and study some of
their properties in the more general case.

\begin{Definition} \label{def:finite-type-injectives}
  Assume $Q$ is a quiver of finite type.
  For $i \in Q_0$, let $q^i = \Hom_\C (e_i \cP, \C)$ be
  the dual space of the right $\cP$-module $e_i \cP$.  Define a left
  $\cP$-module structure on $q^i$ by setting $a \cdot f(x) = f(xa)$,
  for $a \in \cP$, $f \in q^i$, and $x \in e_i \cP$.
\end{Definition}

\begin{Lemma} \label{lem:injectives}
  If $Q$ is a quiver of finite type, then $\{q^i\}_{i \in Q_0}$ is a
  set of representatives
  of the isomorphism classes of indecomposable injective $\cP$-modules.
  Furthermore, $q^i$ is an injective hull of $s^i$.
\end{Lemma}

\begin{proof}
If $Q$ is of finite type, then $\cP$ is finite-dimensional by Proposition~\ref{prop:path-algebra-type}.  The result then follows from Lemma~\ref{lem:projectives} and a well-known fact about modules over finite-dimensional algebras (see, for example, \cite[Cor.~3.66]{Lam99}).
\end{proof}

For $w = \sum_i w_i i \in \N Q_0$, define the semi-simple
$\cP$-module
\[
  s^w = \bigoplus_{i \in Q_0} (s^i)^{\oplus w_i}.
\]
Let $q^i$ be the injective hull of $s^i$ in the category $\cP$-Mod
(if $Q$ is a quiver of finite type, this agrees with the notation of
Definition~\ref{def:finite-type-injectives}). Then
\[
  q^w = \bigoplus_{i \in I} (q^i)^{\oplus w_i}
\]
is the injective hull of $s^w$.  We also define
\[
  p^w = \bigoplus_{i \in I} (p^i)^{\oplus w_i}.
\]

%

\stepcounter{Theorem}

\begin{Remark} \label{rem:p-nilpotent}
  It follows from Proposition~\ref{prop:path-algebra-type}, Lemma~\ref{lem:injectives},
  and Proposition~\ref{prop:injective=projective} that if $Q$ is a
  quiver of finite type, then $p^w$ (and $q^w$)
  is nilpotent.  However, in
  general the $p^w$ are not nilpotent.
\end{Remark}

\begin{Proposition} \label{prop:injective-ln-condition}
If $Q$ is of affine (tame) type, then $q^w$ is locally nilpotent for
all $w \in \N Q_0$. If $Q$ is connected and of indefinite (wild) type, then $q^w$
is not locally nilpotent for any $w \in \N Q_0$, $w \ne 0$.
\end{Proposition}

The following proof was explained to us by W.~Crawley-Boevey.

\begin{proof}
It suffices to consider the case where $w=i$ for some $i \in Q_0$.  We identify $\Z Q_0$ with the root lattice via $\sum v_j j \leftrightarrow \sum v_j \alpha_j$.  We first assume that $Q$ is connected of wild type. Let $\beta$ be a minimal positive imaginary root. Thus $(\beta, j) \le 0$ for all $j \in Q_0$. Suppose
the support of $\beta$ is all of $Q_0$. Since $Q$ is wild, $\beta$
cannot be a radical vector (see \cite[Theorem~4.3]{K}), so $(\beta,
j) < 0$ for some $j \in Q_0$. If, on the other hand, the support of
$\beta$ is not all of $Q_0$, we take $j \in Q_0$ to be a vertex not
in the support of $\beta$ but connected to it by an arrow and we
again have $(\beta, j) < 0$.  By \cite[Theorem~1.2]{Cra01}, there is
a simple module $T$ for the preprojective algebra of dimension
$\beta$. By \cite[Lemma~1]{Cra00}, $\Ext^1 (T,s^j)$ is nonzero.  Let
$V$ be a nontrivial extension of $T$ by $s^j$.  This module must
embed in the injective hull $q^j$ of $s^j$ and thus $q^j$ cannot be
locally nilpotent.  Thus the result holds whenever $(\beta,i) <0$.  For general $i$, choose a shortest path from $i$ to some $j$ with $(\beta,j)<0$ and consider the corresponding nilpotent module $U$ with head $s^j$ and socle $s^i$.  Then, as above, there is a nontrivial extension of $T$ by $U$, which must embed into $q^i$. So $q^i$ is not locally nilpotent.

Now assume that $Q$ is of tame type.  Since the preprojective
algebra of a tame quiver is a finitely generated $\C$-algebra,
noetherian, and a polynomial identity ring \cite[Theorem~6.5]{BGL87}
(see \cite{Rin98} for a proof that the preprojective algebra
considered there is the same as the one considered here), any simple
module is finite-dimensional (see \cite[Theorem~13.10.3]{MR01}).  By
\cite[Theorem~2]{Jat76}, the injective hull of a simple $\cP$-module
is artinian. In particular, finitely generated submodules of
injective hulls of simple modules are artinian and noetherian.  Thus
they are of finite length and hence finite-dimensional.  Now, the
dimension vectors of simple $\cP$-modules are the coordinate vectors
$i \in Q_0$ and the minimal imaginary root $\delta$.  Since
$(\delta, i) = 0$ for all $i \in Q_0$, there are no nontrivial
extensions between simples of dimension $\delta$ and the
one-dimensional simples. Therefore, the composition factors of the
finite-dimensional submodules of the injective hull $q^i$ of $s^i$
are all one-dimensional simple modules.  Thus $q^i$ is locally
nilpotent.
\end{proof}

\begin{Remark}
In types $A$ and $D$, there exist simple and explicit descriptions
of the representations $q^i$,  $i \in Q_0$, in terms of classical
combinatorial objects such as Young diagrams (see
\cite{FS03,Sav03b,S03}).  This allows one to give simple and
explicit descriptions of the injective modules $q^w$ for any $w \in
\N Q_0$ when the underlying graph of the corresponding quiver is of
type $A$ or $D$.
\end{Remark}

\subsection{Quiver grassmannians}

\begin{Definition}[Quiver grassmannian] \label{def:QG}
For a $\C Q$-module $V$, let $\Gr_Q(V)$ be the set of all finite-dimensional $\C
Q$-submodules of $V$ and let $\NGr_Q(V)$ be the set of all finite-dimensional nilpotent $\C Q$-submodules of $V$.  We have natural decompositions
\begin{gather*}
    \Gr_Q(V) = \bigsqcup_{u \in \N Q_0} \Gr_Q(u,V),\quad
    \Gr_Q(u,V) = \{U \in \Gr_Q (V)\ |\ \dim U = u\}, \\
    \NGr_Q(V) = \bigsqcup_{u \in \N Q_0} \NGr_Q(u,V),\quad
    \NGr_Q(u,V) = \{U \in \NGr_Q (V)\ |\ \dim U = u\}.
\end{gather*}
We call $\Gr_Q(u,V)$ a \emph{quiver grassmannian} and $\NGr_Q(u,V)$ a \emph{nilpotent quiver grassmannian}.
If $V$ is a finite-dimensional module, then $\Gr_Q(u,V)$ is a closed subset of the usual grassmannian of dimension $u$ subspaces of $V$ and thus is a projective variety.  More generally, if the linear span of all points in $\Gr_Q(V)$ or $\NGr_Q(V)$ is finite dimensional in $V$, then that set has a natural variety structure.

If $V$ is a $\cP$-module, then $\cP$-submodules of $V$ are the same as $\C \tilde Q$-submodules of $V$.  Hence one can think of
$\Gr_{\tilde Q}(V)$ as the set of all finite-dimensional $\cP$-submodules of $V$.
Therefore, we will often write $\Gr_\cP (V)$ and $\Gr_\cP (u,V)$ for
$\Gr_{\tilde Q}(V)$ and $\Gr_{\tilde Q}(u,V)$ when $V$ is a
$\cP$-module.  Similarly, we will sometimes write $\NGr_\cP (V)$ and $\NGr_\cP (u,V)$ for
$\NGr_{\tilde Q}(V)$ and $\NGr_{\tilde Q}(u,V)$.
\end{Definition}

\begin{Example}[Grassmannians]
If $Q$ is the quiver with a single vertex and no arrows, then $\cP =
\C$ and $\cP$-modules are simply vector spaces.  Then $\Gr_\cP(u,V)
= \Gr (u,V)$ is the usual grassmannian of dimension $u$ subspaces of
$V$.
\end{Example}

\begin{Example}[Partial flag varieties]
Suppose $Q$ is the quiver with $Q_0=\{1,2,\dots,n\}$ and $Q_1=\{a_1,\dots,a_{n-1}\}$, where $s(a_i)=i$, $t(a_i)=i+1$ for all $i=1,\dots,n-1$.  Fix a positive integer $d$ and set $V_i=\C^d$ for all $i=1,\dots,n$.  For each $1 \leq i \leq n-1$, let $a_i$ act by the identification $V_i \cong V_{i+1}$.  Then for $u \in \N Q_0$ with $u_1 \le u_2 \le \dots \le u_n \le d$, the quiver grassmannian $\Gr_\cP(u,V)$ is isomorphic to the partial flag variety
\[
  \{0 \subseteq F_1 \subseteq F_2 \subseteq \dots \subseteq F_n \subseteq \C^d\ |\ \dim F_i = u_i\}.
\]
\end{Example}

\begin{Definition} \label{def:Aut-action-on-QG}
For $V \in \cP$-Mod, we define a natural action of $\Aut_\cP V$ on
$\Gr_\cP(u,V)$ given by
\[
  (g, U) \mapsto g (U),\quad g \in \Aut_\cP V,\quad U \in
  \Gr_\cP(u,V).
\]
\end{Definition}


\section{Quiver varieties} \label{sec:QVs}

In this section we briefly recall certain quiver varieties defined by Lusztig
and Nakajima, referring the reader to \cite{L91,Nak94,Nak98} for
further details, as well as the Demazure quiver varieties introduced
by the first author in \cite{Sav04a}.  We fix a quiver $Q=(Q_0,Q_1)$
and let $\cP=\cP(Q)$ denote its preprojective algebra.

\subsection{Lusztig and Nakajima quiver varieties}

For $V \in \cP_0$-mod, define
\[
  \Rep_{\tilde Q} V = \bigoplus_{a \in \tilde Q_1} \Hom_\C (V_{s(a), t(a)}).
\]
For a path $\beta = a_l \cdots a_1$ in $Q$ and $x=(x_a)_{a \in
\tilde Q_1} \in \Rep_{\tilde Q} V$, we define $x_\beta = x_{a_l}
\cdots x_{a_1}$.  For an element $\sum_j c_j \beta_j \in \C Q$, we
define
\[
  x_{\sum_j c_j \beta_j} = \sum_j c_j x_{\beta_j}.
\]
Thus each $x \in \Rep_{\tilde Q} V$ defines a representation $\C \tilde Q
\to \End_\C V$ of graded dimension $\dim_{Q_0} V$ (i.e.\ whose
induced representation of $(\C Q)_0$ is in the isomorphism class
determined by $\dim_{Q_0} V$).  Furthermore, each such
representation comes from an element of $x \in \Rep_{\tilde Q} V$.
These two statements are simply the equivalence of categories
between the representations of the quiver and of the path algebra.
We say that $x$ is \emph{nilpotent} if there exists $N
> 0$ such that $x_\beta=0$ for all paths $\beta$ of length greater
than $N$.

\begin{Definition}[Lusztig nilpotent variety]
For $V \in \cP_0$-mod, define $\Lambda(V)=\Lambda_Q(V)$ to be the
set of all nilpotent $\mathcal{P}$-module structures on $V$
compatible with its $\cP_0$-module structure. More precisely,
\[
  \Lambda(V) = \left\{ x \in \Rep_{\tilde Q} V\
  \left| \ \sum_{a \in Q_1,\ t(a)=i} x_a x_{\bar a} - \sum_{a \in Q_1,\
  s(a)=i} x_{\bar a} x_a = 0 \ \forall\ i \in Q_0,\
  x \text{ nilpotent} \right. \right\}.
\]
We call $\Lambda(V)$ a \emph{Lusztig nilpotent variety}.
\end{Definition}

As above, elements of $\Lambda(V)$ are in natural one-to-one
correspondence with nilpotent representations $\cP \to \End_\C V$ of
graded dimension $\dim_{Q_0} V$.

For $V,W \in \cP_0$-mod, let $\Lambda(V,W) = \Lambda(V) \times
\Hom_{\cP_0}(V,W)$.  We say that $(x,t) \in \Lambda(V,W)$ is
\emph{stable} if there exists no non-trivial $x$-invariant
$\cP_0$-submodule of $V$ contained in $\ker t$.  This is equivalent
to the condition that $\ker ((x,t)|_{V_i}) = 0$ for all $i \in Q_0$
(see \cite[Lemma~3.4]{FS03} -- while the statement there is for type
$A$, the proof carries over to the more general case).  We denote
the set of stable elements by $\Lambda(V,W)^\st$. There is a natural
action of $\GL_V$ on $\Lambda(V,W)$ and the restriction to
$\Lambda(V,W)^\st$ is free (see \cite{Nak94,Nak98}). We denote the
$\GL_V$-orbit through a point $(x,t)$ by $[x,t]$.

\begin{Definition}[Lagrangian Nakajima quiver variety]
For $V, W \in \cP_0$-mod, let $\mathfrak{L}(V,W) = \Lambda(V,W)^\st
/ \GL_V$. We call $\mathfrak{L}(V,W)$  a \emph{lagrangian
Nakajima quiver variety}.  Up to isomorphism, this variety depends
only on $v=\dim_{Q_0} V$ and $w=\dim_{Q_0} W$ and so we will
sometimes denote it by $\mathfrak{L}(v,w)$.
\end{Definition}

\begin{Remark}
The quiver varieties defined above are lagrangian subvarieties of
what are usually called the Nakajima quiver varieties
\cite{Nak94,Nak98}.
\end{Remark}

\subsection{Group actions} \label{subsec:quiver-action}

Let $G_\cP$ be the group of algebra automorphisms of $\cP$ that fix
$\cP_0$. The group $\GL_W$ acts naturally on $\Hom_{\cP_0} (V,W)$. As
above, we identify elements of $\Lambda(V)$ with nilpotent
representations $\cP \to \End_\C V$ of graded dimension $\dim_{Q_0}
V$.  Then
\[
  (h,(x,t)) \mapsto (h \star x,t),\quad h \star x = x \circ h^{-1},
  \quad h \in G_\cP,
\]
defines a $G_\cP$-action on $\Lambda(V,W)$.  The actions of $\GL_W$
and $G_\cP$ commute and both commute with the $\GL_V$-action. Since they also preserve the stability condition, they
define a $\GL_W \times G_\cP$-action on $\mathfrak{L}(v,w)$.

We can use this action to define $\GL_W \times \C^*$-actions on
$\mathfrak{L}(v,w)$ as follows.  Suppose a function $m: \tilde Q_1 \to \Z$
is given such that $m(a) = -m(\bar a)$ for all $a \in \tilde Q_1$. Then the
map $a \mapsto z^{m(a)+1} a$, $z \in \C^*$, extends to an
automorphism of $\cP$ fixing $\cP_0$. We denote this automorphism by
$h_m (z)$.  Thus $h_m$ defines a group homomorphism $\C^* \to
G_\cP$. Then the homomorphism
\begin{equation} \label{eq:GC-to-GP-homom}
  \GL_W \times \C^* \to \GL_W \times G_\cP,\quad (g,z) \mapsto
  (z g, h_m (z))
\end{equation}
defines a $\GL_W \times \C^*$-action on $\mathfrak{L}(v,w)$ which we
denote by $\star_m$.

We give two important examples of this action (see
\cite[\S2.7]{Nak01b} and \cite{Nak04}).  First, for each pair $i,j
\in Q_0$ connected by at least one edge, let $b_{ij}$ denote the
number of arrows in $Q_1$ joining $i$ and $j$. We fix a numbering
$a_1, \dots, a_{b_{ij}}$ of these arrows, which induces a numbering
$\bar a_1, \dots, \bar a_{b_{ij}}$ of the corresponding arrows in
$\bar Q_1$. Define $m_1 : H \to \Z$ by
\[
  m_1(a_p) = b_{ij} + 1 - 2p,\quad m_1(\bar a_p) = -b_{ij}-1+2p.
\]
For the second action, we define $m_2(a)=0$ for all $a \in Q_1$.

\subsection{Demazure quiver varieties}

Let $\g$ be the Kac-Moody algebra corresponding to the underlying
graph of $Q$ (i.e.\ whose Dynkin diagram is this graph) and let
$\mathcal{W}$ be its Weyl group.  Recall that $\mathcal{W}$ acts
naturally on the weight lattice of $\g$. For $u \in \Z Q_0$, we
define elements of the weight and root lattice by
\[
  \omega_u = \sum_{i \in Q_0} u_i \omega_i,\quad \alpha_u = \sum_{i
  \in Q_0} u_i \alpha_i.
\]

\begin{PropDef}[{\cite[Proposition~5.1]{Sav04a}}]
\label{propdef:extremal-QV} The lagrangian Nakajima quiver variety
$\mathfrak{L}(v,w)$ is a point if and only if $\omega_w - \alpha_v =
\sigma (\omega_w)$ for some $\sigma \in \mathcal{W}$ (i.e.\ $\omega_w
- \alpha_v$ is an extremal weight of the irreducible representation of highest weight $\omega_w$, equivalently $v$ is $w$-extremal in the sense of Definition~\ref{def:w-extremal}). In this case, we let
$(x^{w,\sigma},t^{w,\sigma})$ be a representative (unique up to
isomorphism) of the $\GL_V$-orbit corresponding to this
point. So $\mathfrak{L}(v,w) = \{[x^{w,\sigma}, t^{w,\sigma}]\}$ when $\omega_w - \alpha_v = \sigma(\omega_w)$.
\end{PropDef}

\begin{Definition}[Demazure quiver variety] \label{def:DQV}
For $\sigma \in \mathcal{W}$ and $v,w \in \N Q_0$, let
$\mathfrak{L}_\sigma(v,w)$ be the subvariety consisting of all
$[x,t] \in \mathfrak{L}(v,w)$ such that $(x,t)$ is isomorphic to a
subrepresentation of $(x^{w,\sigma}, t^{w,\sigma})$. We call
$\mathfrak{L}_\sigma(v,w)$ a \emph{Demazure quiver variety}.
\end{Definition}

\begin{Remark}
It follows from the uniqueness assertion in Proposition/Definition~\ref{propdef:extremal-QV} that the $\GL_W \times G_\cP$-action on $\mathfrak{L}(v,w)$ fixes $\mathfrak{L}_\sigma(v,w)$ for all $\sigma \in \mathcal{W}$.  Thus we have an induced $\GL_W \times G_\cP$-action on the Demazure quiver varieties.
\end{Remark}


\section{Quiver varieties as quiver grassmannians}
\label{sec:QVs-and-QGs}

\subsection{Lagrangian Nakajima quiver varieties as quiver grassmannians} \label{LNasQG}

In this section we show that certain quiver grassmannians are
homeomorphic to the lagrangian Nakajima quiver varieties. We begin
with a key technical proposition.

\begin{Proposition} \label{prop:unique-injection}
  Suppose $A = \bigoplus_{n \ge 0} A_n$ is a graded algebra and $V$ is
  a locally nilpotent $A$-module.
  Furthermore, suppose $S$ is a
  semisimple locally nilpotent $A$-module with injective hull $E$.

  \begin{enumerate}
  \item \label{prop-item1:unique-injection}
  Let $\pi : E \to S$ be an $A_0$-linear retract for the canonical embedding $\iota : S \to E$ (that is, an $A_0$-linear map such that $\pi \iota = \id$) and let $\tau :
  V \to S$ be a homomorphism of $A_0$-modules.  Then
  there exists a unique $A$-module homomorphism $\gamma : V \to E$ such
  that the following diagram commutes:
  \[
    \xymatrix{
      & E \ar[d]^\pi \\
      V \ar[r]^\tau \ar[ur]^\gamma & S
    }
  \]
  Furthermore, the map $\gamma$ is injective if and only if $\tau|_{\socle V}$
  is injective.

  \item \label{prop-item2:unique-injection}
  Suppose $\pi_1, \pi_2 : E \to S$ are $A_0$-linear retracts for the canonical embedding $\iota : S \to E$.  Then there exists
  a unique $\gamma \in \Aut_A E$ such that
  $\pi_2 = \pi_1 \gamma$.  The map $\gamma$ fixes $S$ pointwise.
  Conversely, given an $A_0$-linear retract
  $\pi: E \to S$ and any $\gamma \in \Aut_A E$ fixing $S$
  pointwise, $\pi \gamma : E \to S$ is also a $A_0$-linear retract.
  \end{enumerate}
\end{Proposition}

\begin{proof}
  Since $V$ is locally nilpotent, we have a filtration
  \[
    0 = V^{(0)} \subseteq V^{(1)} = \socle V \subseteq V^{(2)} \subseteq
    V^{(3)} \subseteq \dots
  \]
  of $V$ where $V^{(n)} = \{m \in V \ |\ A_{\ge n} \cdot m = 0\}$.  We prove by
  induction on $n$ that there exists a unique homomorphism $\gamma_n:
  V^{(n)} \to E$ such that the diagram
  \begin{equation} \label{eq:induction-diagram}
    \xymatrix{
      & E \ar[d]^\pi \\
      V^{(n)} \ar[r]^(0.6){\tau_n} \ar[ur]^{\gamma_n} & S
    }
  \end{equation}
  commutes, where $\tau_n = \tau|_{V^{(n)}}$.
  Since  $V^{(1)} = \socle V$ and $A_+ \cdot \socle V = 0$,
  we must have $\gamma_1(V^{(1)}) \subseteq S$ and
  so the unique choice for $\gamma_1$ is $\tau_1$.
  Suppose the statement holds for $n=k$.  Since $E$ is injective, there
  exists an $A$-module homomorphism $\hat \gamma_{k+1}$ such that the following diagram
  commutes:
  \[
    \xymatrix{
      V^{(k+1)} \ar[r]^{\hat \gamma_{k+1}} & E \\
      V^{(k)} \ar@{^(->}[u] \ar[ur]_{\gamma_k} &
    }
  \]
  Define $\gamma_{k+1}$ by
  \[
    \gamma_{k+1} = \hat \gamma_{k+1} - \pi \circ \hat \gamma_{k+1}
    + \tau.
  \]
  It is then clear that the diagram \eqref{eq:induction-diagram}
  commutes (with $n=k+1$).  Note also that $\gamma_{k+1}|_{V^{(k)}} = \gamma_k$.
  We claim that $\gamma_{k+1}$ is a
  homomorphism of $A$-modules.  Since it is an $A_0$-module homomorphism
  by definition, it
  suffices to show it commutes with the action of $A_+$.

  For $r \in A_+$ and $m
  \in V^{(k+1)}$, we have $r \cdot m \in V^{(k)}$.  Also, $A_+ \cdot S = 0$.
  Then
  \begin{align*}
    r \cdot \gamma_{k+1}(m) &= r \cdot (\hat \gamma_{k+1}(m) - \pi \circ \hat \gamma_{k+1}(m)
    + \tau(m)) \\
    &= r \cdot \hat \gamma_{k+1}(m)
    = \hat \gamma_{k+1}(r \cdot m)
    = \gamma_{k}(r \cdot m) \\
    &= \gamma_{k+1}(r \cdot m)
  \end{align*}
  as desired.

  Now suppose that $\gamma_{k+1}'$ is another $\cP$-module
  homomorphism making
  \eqref{eq:induction-diagram} commute (with $n=k+1$).
  By the inductive hypothesis, we have
  $\gamma_{k+1}|_{V^{(k)}} = \gamma_{k+1}'|_{V^{(k)}}$. For all
  $r \in A_+$ and $m \in V^{(k+1)}$, we have
  \[
    r \cdot \gamma_{k+1} (m) = \gamma_{k+1} (r \cdot m) =
    \gamma_{k+1}'(r \cdot m) = r \cdot \gamma_{k+1}' (m).
  \]
  Thus $\gamma_{k+1}(m) - \gamma_{k+1}'(m)$ lies in $S$.  Therefore
  \[
    \gamma_{k+1}(m) - \gamma_{k+1}'(m) = \pi(\gamma_{k+1}(m) -
    \gamma_{k+1}'(m)) = \pi(\gamma_{k+1}(m)) - \pi(\gamma_{k+1}'(m))
    = \tau(m) - \tau(m) = 0.
  \]
  The induction is
  complete and we obtain the desired map $\gamma$ by taking the
  limit.

  Note that $\gamma|_{\socle V}
  =\tau|_{\socle V}$.  Since a homomorphism of modules is injective
  if and only if its restriction to the socle is injective, it
  follows that $\gamma$ is injective if and only if $\tau|_{\socle
  V}$ is injective.

  We now prove \eqref{prop-item2:unique-injection}.
  By \eqref{prop-item1:unique-injection}, there exists a unique
  $A$-module homomorphism $\gamma : E \to E$ such that $\pi_2 =
  \pi_1 \gamma$.  Similarly, there exists a unique $A$-module
  automorphism $\tilde \gamma : E \to E$ such that $\pi_1 = \pi_2
  \tilde \gamma$ and $\gamma \tilde \gamma = \tilde \gamma \gamma =
  \id$ by the uniqueness assertion in
  \eqref{prop-item1:unique-injection}.  Thus $\gamma$ is an
  $A$-automorphism of $E$.  The converse statement is trivial.
\end{proof}

\begin{Remark}
The retract $\pi : E \to S$ in Proposition~\ref{prop:unique-injection}
is equivalent to choosing an $A_0$-module decomposition $E = S
\oplus T$.  The second part of the proposition states that any two
such decompositions are related by a unique $A$-module automorphism
of $E$ fixing $S$.
\end{Remark}

\begin{Definition}
Let $V$ be a $\mathcal{P}_0$-module of graded dimension $v$. We
define $\widehat {\NGr}_\cP(v,q^w)$ to be the set of injective
$\mathcal{P}_0$-module homomorphisms $\gamma : V \to q^w$ whose
image is a nilpotent $\mathcal{P}$-submodule of $q^w$.
\end{Definition}

\begin{Theorem} \label{thm:QG=QV}
Fix $v,w \in \N Q_0$.  Then $\widehat {\NGr}_\cP(v,q^w)$ and $\NGr_\cP(v,q^w)$ are naturally algebraic varieties.  Furthemore, there is a bijective $\GL_V$-equivariant algebraic map from $\widehat {\NGr}_\cP(v,q^w)$ to $\Lambda(v,w)^\st$ and a bijective algebraic map from
$\NGr_\cP(v,q^w)$ to $\mathfrak{L}(v,w)$. In particular, $\widehat
{\NGr}_\cP(v,q^w)$ is homeomorphic to $\Lambda(v,w)^\st$ and
$\NGr_\cP(v,q^w)$ is homeomorphic to $\mathfrak{L}(v,w)$.
\end{Theorem}

The proof of Theorem~\ref{thm:QG=QV} will be given at the end of Section~\ref{subsec:demazure-QG}.

\begin{Remark} \label{rem:Lusztig-conversation}
Lusztig \cite{Lus98,Lus00b} has described a canonical bijection
between the lagrangian Nakajima quiver varieties and grassmannian
type varieties inside the projective modules $p^w$ (see
Section~\ref{sec:injective-versus-projective}). In several places in
the literature, it was claimed that the varieties defined by Lusztig
are isomorphic (as algebraic varieties) to the lagrangian Nakajima
quiver varieties.  However, the authors were not aware of a proof
existing in the literature.  Most references for this statement were
to Lusztig's papers \cite{Lus98,Lus00b}, where the points of the two
varieties are shown to be in canonical bijection (similar to the
situation in the current paper).  Lusztig informed the authors
that he was not aware of a proof that the varieties are isomorphic.
After the appearance of an earlier version of the current paper \cite{ST09}, Shipman~\cite{Shi10} proved that the varieties are indeed isomorphic.  From now on, we will incorporate Shipman's work as it allows us to strengthen several results; in particular (see Corollary~\ref{cor:QG=QV-variety-map} in the Appendix) the map $\bar \iota$ in the proof below is an isomorphism of algebraic varieties.
\end{Remark}

\begin{Remark} \mbox{}
\begin{enumerate}
  \item The role of the retract $\pi$ in
  Proposition~\ref{prop:unique-injection} is to ensure the
  uniqueness of
  $\gamma$.

  \item When $Q$ is of finite type, the injective module $q^w$ is also projective (see
  Proposition~\ref{prop:injective=projective}) and thus   Theorem~\ref{thm:QG=QV} follows from \cite[\S2.1]{Lus00b}.

  \item The isomorphisms of Theorem~\ref{thm:QG=QV} depend on
  the choice of retract $\pi : q^w \to s^w$.  By
  Proposition~\ref{prop:unique-injection}\eqref{prop-item2:unique-injection},
  isomorphisms coming from different retracts are related by
  an automorphism of $q^w$ fixing $s^w$.

  \item In Lusztig's grassmannian type realization of the lagrangian
  Nakajima quiver varieties \cite{Lus98,Lus00b}, one must require
  that the submodules contain all paths of large enough length (this
  corresponds to the nilpotency condition in the definition of the quiver
  varieties).  In the current approach using injective
  modules, this corresponds to requiring that submodules be nilpotent.
\end{enumerate}
\end{Remark}

\subsection{Demazure quiver grassmannians} \label{subsec:demazure-QG}

As before, let $\g$ be the Kac-Moody algebra corresponding to the
underlying graph of $Q$ and let $\mathcal{W}$ be its Weyl group with
Bruhat order $\preceq$.

\begin{Definition} \label{def:w-extremal}
  For each $w \in \N Q_0$, we define an action of $\mathcal{W}$ on $\Z Q_0$
  as follows.  For $v \in \Z Q_0$ and $\sigma \in \mathcal{W}$, define
  $\sigma \cdot_w v = u$ where $u$ is the unique element of $\Z
  Q_0$ satisfying
  \[
    \sigma (\omega_w - \alpha_v) = \omega_w - \alpha_u.
  \]
  We say that $v \in \N Q_0$ is \emph{$w$-extremal} if $v \in \mathcal{W}
  \cdot_w 0$.
\end{Definition}

\begin{Lemma}
  If $v,w \in \N Q_0$ and $\omega_w - \alpha_v$ is a weight
  of the irreducible highest weight representation of $\g$ of
  highest weight $\omega_w$
  (i.e the corresponding weight space is nonzero), then $\sigma \cdot_w v \in \N Q_0$
  for all $\sigma \in \mathcal{W}$.  In particular $\mathcal{W}
  \cdot_w 0 \subseteq \N Q_0$.
\end{Lemma}

\begin{proof}
  This follows easily from the fact that $\mathcal{W}$ acts on the
  weights of highest weight irreducible representations and the weight multiplicities
  are invariant under this action.
\end{proof}

\begin{Proposition} \label{prop:extremal}
  For $v \in \N Q_0$, the following statements are equivalent:
  \begin{enumerate}
    \item \label{extremal-item1} $v$ is $w$-extremal,
    \item \label{extremal-item2} $\mathfrak{L}(v,w)$ consists of a single point,
    \item \label{extremal-item3} $\NGr_\cP(v,q^w)$ consists of a single point, and
    \item \label{extremal-item4} there is a unique nilpotent submodule of $q^w$ of graded
    dimension $v$.
  \end{enumerate}
\end{Proposition}

\begin{proof}
  The equivalence of \eqref{extremal-item1} and
  \eqref{extremal-item2} is given in \cite[Proposition~5.1]{Sav04a}.
  The equivalence of \eqref{extremal-item2}, \eqref{extremal-item3}
  and \eqref{extremal-item4} follows from
  Theorem~\ref{thm:QG=QV}.
\end{proof}

\begin{Definition}[Demazure quiver grassmannian] \label{def:DQG}
  For $\sigma \in \mathcal{W}$, we let $q^{w,\sigma}$
  denote the unique nilpotent submodule of $q^w$ of graded dimension $\sigma \cdot_w
  0$.  We call $\Gr_\cP(v,q^{w,\sigma})$ a \emph{Demazure quiver
  grassmannian}.
\end{Definition}

\begin{Proposition} \label{prop:Demazure-nesting}
  If $\sigma_1, \sigma_2 \in \mathcal{W}$ with $\sigma_1 \preceq \sigma_2$,
  then $q^{w,\sigma_2}$
  has a unique submodule of graded dimension $\sigma_1 \cdot_w 0$ and this
  submodule is isomorphic to $q^{w,\sigma_1}$.
\end{Proposition}

\begin{proof}
  Since $\sigma_1 \preceq \sigma_2$, we have $L_{\omega_w,\sigma_1}
  \subseteq L_{\omega_w,\sigma_2}$, where $L_{\omega_w,\sigma_i}$ is the
  Demazure module corresponding to $L_{\omega_w}$ (the irreducible integrable
  highest weight $\g$-module with highest weight $\omega_w$) and $\sigma_i$.
  It then follows from \cite[Theorem~7.1]{Sav04a} that
  $q^{w,\sigma_1}$ is (isomorphic to) a submodule of
  $q^{w,\sigma_2}$.
  Since any submodule of $q^{w,\sigma_2}$ is also a submodule of $q^w$,
  uniqueness follows immediately from
  Proposition~\ref{prop:extremal}.
\end{proof}

\begin{Proposition} \label{prop:DQG=DQV}
  Fix $\sigma \in \mathcal{W}$ and $v,w \in \N Q_0$.  Then
  $\Gr_\cP(v,q^{w,\sigma})$ is isomorphic (as an algebraic variety) to the Demazure quiver
  variety $\mathfrak{L}_\sigma(v,w)$.
\end{Proposition}

\begin{proof}
  This follows immediately from Definitions~\ref{def:DQV} and~\ref{def:DQG}
  and the description of the homeomorphism $\NGr_\cP(v,q^w) \cong
  \mathfrak{L}(v,w)$ given in Theorem~\ref{thm:QG=QV}, which is actually an
  isomorphism of algebraic varieties by Corollary~\ref{cor:QG=QV-variety-map}.
\end{proof}

\begin{Remark}
  Note that if $Q$ is a quiver of finite type and $\sigma_0$
  is the longest element of $\mathcal{W}$, then
  $\mathfrak{L}_{\sigma_0}(v,w) = \mathfrak{L}(v,w)$ and
  $\Gr(v,q^{w,\sigma_0}) = \Gr(v,q^w)$ for all $v, w \in
  \N Q_0$.
\end{Remark}

The $(q^{w,\sigma})_{\sigma \in \mathcal{W}}$ form a directed system
under the Bruhat order.  Let $\tilde q^w$ be the direct limit of
this system.

\begin{Lemma} \label{lem:loc-nil-contained-in-tqw}
Any locally nilpotent submodule $V$ of $q^w$ is contained in $\tilde
q^w$.
\end{Lemma}

\begin{proof}
First note that for $n \in \N$, the submodule $(q^w)^{(n)} = \{v \in
q^w \ |\  \mathcal{P}_{\ge n} \cdot v = 0\}$ of $q^w$ is
finite-dimensional. This follows from the fact that $q^i$ is a
submodule of $\Hom_\C (e_i \mathcal{P},\C)$ (since this is an
injective module containing $s^i$), which has this property, and
$q^w = \bigoplus_{i \in I} (q^i)^{\oplus w_i}$.

Since $V$ is locally nilpotent, we have a filtration
\[
  0 = V^{(0)} \subseteq V^{(1)} = \socle V \subseteq V^{(2)} \subseteq \dots
\]
where $V^{(n)} = \{v \in V \ |\  \mathcal{P}_{\ge n} \cdot v = 0\}$.
Local nilpotency of $V$ ensures that $\bigcup_n V^{(n)} =V$.  It
suffices to show that each $V^{(n)}$ is contained in $\tilde q^w$.
Since $V^{(n)} \subseteq (q^w)^{(n)}$, it follows that $V^{(n)}$ is
finite-dimensional. Choose a linear retract $\pi : q^w \to
s^w$.  By Theorem~\ref{thm:QG=QV}, $V$ corresponds to a point of
$\mathfrak{L}(v,w)$.  Choose $\sigma \in \mathcal{W}$ sufficiently
large so that the $(\omega_w - \alpha_v)$-weight space of the
representation $L_{\omega_w}$ is contained in the Demazure module
$L_{\omega_w, \sigma}$ (we can always do this since the weight space
is finite-dimensional).  Then by Proposition~\ref{prop:DQG=DQV}, we
have that $V \subseteq q^{w,\sigma} \subseteq \tilde q^w$.
\end{proof}

\begin{Theorem} \label{thm:ln-injective-hull}
We have that $\tilde q^w$ is the injective hull of $s^w$ in the
category $\cP$-lnMod.
\end{Theorem}

\begin{proof}
Since each $q^{w,\sigma}$ is nilpotent, it follows that $\tilde q^w$
is locally nilpotent and thus belongs to the category $\cP$-lnMod.
Furthermore, it is clear that $\tilde q^w$ has socle $s^w$ and that
it is an essential extension of $s^w$. It remains to show that
$\tilde q^w$ is an injective object of $\cP$-lnMod. Suppose $M$ and
$N$ are locally nilpotent $\mathcal{P}$-modules and we have a
homomorphism $M \to \tilde q^w$ and an injection $M \hookrightarrow
N$. Since $q^w$ is injective in the category of
$\mathcal{P}$-modules, there exists a homomorphism $h : N \to q^w$
such that the following diagram commutes:
\[
  \xymatrix{
    N \ar[rrd]^{h} & & \\
    M \ar@{^(->}[u] \ar[r] & \tilde q^w \ar@{^(->}[r] & q^w
  }
\]
Since $N$ is locally nilpotent, $h(N)$ is a locally nilpotent
submodule of $q^w$.  Therefore the map $h$ factors through $\tilde
q^w$ by Lemma~\ref{lem:loc-nil-contained-in-tqw}.
\end{proof}

\begin{Corollary}
We have that $\tilde q^w \cong q^w$ if and only if $Q$ is of finite
or affine (tame) type.
\end{Corollary}

\begin{proof}
This follows immediately from Theorem~\ref{thm:ln-injective-hull}
and Proposition~\ref{prop:injective-ln-condition}.
\end{proof}

We see from the above that $\{q^{w,\sigma}\}_{\sigma \in
\mathcal{W}}$ is a ``rigid'' filtration of $\tilde q^w$ (rigid in
the sense of the uniqueness of submodules of the given $w$-extremal
graded dimensions). Proposition~\ref{prop:DQG=DQV} can be seen as a
representation theoretic interpretation of this filtration.  It
corresponds to the filtration by Demazure modules of the irreducible
highest-weight representation of $\g$ of highest weight $\omega_w$.
If the quiver $Q$ is of finite type, the Weyl group $\mathcal{W}$,
and hence this filtration, is finite. Otherwise they are infinite.
In the infinite case, we have a filtration of the
infinite-dimensional $\tilde q^w$ by finite-dimensional submodules
$q^{w,\sigma}$, $\sigma \in \mathcal{W}$.

By Lemma~\ref{lem:loc-nil-contained-in-tqw}, we have
\[
  \NGr_\cP(v,q^w) = \Gr_\cP(v,\tilde q^w) \quad \text{and} \quad \widehat{\NGr}_\cP(v,q^w) = \widehat{\Gr}_\cP(v,\tilde q^w).
\]

\begin{proof}[Proof of Theorem~\ref{thm:QG=QV}]
  Any element of $\NGr_\cP(v,q^w)$ is contained in the submodule $\{u \in q^w\ |\ \cP_{\ge k} \cdot u = 0\}$, where $k = \sum_i v_i$.  This submodule is finite-dimensional as in the proof of Lemma~\ref{lem:loc-nil-contained-in-tqw}.  Thus, $\NGr_\cP(v,q^w)$ and $\widehat{\NGr}_\cP(v,q^w)$ are naturally algebraic varieties.

  Fix $V \in \cP_0$-mod of graded dimension $v$ and a
  $\mathcal{P}_0$-module homomorphism $\pi : \tilde q^w \to s^w$ that is the identity on
  $s^w$.  We identify $s^w$ with the $W$ appearing in the definition
  of the quiver varieties. By Lemma~\ref{lem:loc-nil-contained-in-tqw} a point $\gamma \in \widehat \NGr_\cP(v,q^w)$ defines an
  embedding of $V$ into $\tilde q^w$, hence a $\cP$-module
  structure on $V$ satisfying the stability condition
  and so a point of $\Lambda(v,w)^\st$.  More precisely, $\gamma \in \widehat
  \NGr_\cP(v,q^w)$ corresponds to the point $(\gamma^{-1} x^w \gamma, \pi
  \gamma) \in \Lambda(v,w)^\st$, where $x^w$ is the element of $\Rep_{\tilde Q} \tilde q^w$
  corresponding to the $\cP$-module $\tilde q^w$.
  Thus we have a map
  \[
    \iota : \widehat {\NGr}_\cP(v,q^w) \to \Lambda(V,W)^\st,
  \]
  which is clearly algebraic and $\GL_V$-equivariant.
  By Proposition~\ref{prop:unique-injection}, $\iota$
  is bijective.  Passing to the quotient by
  $\GL_V$ we also obtain a bijective algebraic map $\bar \iota$
  from $\NGr_\cP(v,q^w)$ to $\mathfrak{L}(v,w)$.

  Now, $\NGr_\cP(v,q^w)$ and $\mathfrak{L}(v,w)$ are both projective.
  By, for example,
  \cite[Theorem~4.9~and~Exercise~4.4]{Hartshorne:1977}, the image
  of a projective variety under an algebraic map is always
  closed, so $\bar \iota$ takes closed subsets to closed subsets. Since
  $\bar \iota$ is a bijection, this implies that $\bar \iota {}^{{}^{-1}}$ is continuous.
  Hence $\bar \iota$ is a homeomorphism.  Since $\widehat {\NGr}_\cP(v,q^w)$ and
  $\Lambda(v,w)^\st$ are principal $G$-bundles over $\NGr_\cP(v,q^w)$ and
  $\mathfrak{L}(v,w)$, the map $\iota$ also induces a homeomorphism.
\end{proof}

\begin{Remark} \label{rem:tqg}
Since  $\NGr_\cP(v,q^w) = \Gr_\cP(v,\tilde q^w) $, we have also shown that $\Gr_\cP(v, \tilde q^w)$ is isomorphic to $\Lambda(v,w)$. The proof also shows that this isomorphism is uniquly chosen once one fixes a $\mathcal{P}_0$-module
retract $\pi : \tilde q^w \to s^w$.
\end{Remark}


\section{Group actions and graded quiver grassmannians}
\label{sec:actions-and-graded-versions}

In this section we define a natural $\GL_W \times G_\cP$-action on the
quiver grassmannians and show that the maps of
Theorem~\ref{thm:QG=QV} are equivariant.  We then define
graded/cyclic quiver grassmannians and show that they are isomorphic
to the graded/cyclic quiver varieties of Nakajima (see
\cite[\S4.1]{Nak01b} and \cite[\S4]{Nak04}).

\subsection{$\GL_w \times G_\cP$-action and equivariance}

Let $\GL_w=\GL_{s^w}$ and recall that $G_\cP$ is the group of
algebra automorphisms of $\cP$ that fix $\cP_0$ pointwise.  For a
$\cP$-module $V$ and $h \in G_\cP$, denote by $^h V$ the
$\cP$-module with action given by $(a,v) \mapsto h^{-1}(a) \cdot v$.
Now, fix $(g,h) \in \GL_w \times G_\cP$ and a $\cP_0$-module
retract $\pi : \tilde q^w \to s^w$. By
Proposition~\ref{prop:unique-injection} and Lemma~\ref{lem:loc-nil-contained-in-tqw}, there exists a unique
$\mathcal{P}$-module homomorphism $\gamma_{(g,h)} : {^h \tilde q^w} \to
\tilde q^w$ such that the following diagram commutes:
\[
  \xymatrix{
    {^h \tilde q^w} \ar[d]_{\pi} \ar[rr]^{\gamma_{(g,h)}}
    & & \tilde q^w \ar[d]^\pi \\ s^w \ar[rr]^g & &  s^w
  }
\]
The uniqueness assertion of Proposition~\ref{prop:unique-injection} ensures that $\gamma_{(g,h)}$ is bijective with inverse $\gamma_{(g^{-1},h^{-1})}$.  Note that since the action of $\mathcal{P}_0$ on $^h \tilde q^w$ and $\tilde q^w$ is the same, $\gamma_{(g,h)}$ can be considered as a
$\mathcal{P}_0$-automorphism of $\tilde q^w$. This defines a group
homomorphism $\GL_w \times G_\cP \to \GL_{\tilde q^w}$, $(g,h)
\mapsto \gamma_{(g,h)}$. In other words, it defines an action of
$\GL_w \times G_\cP$ on $\tilde q^w$ by $\mathcal{P}_0$-module automorphisms.
This in turn defines an action on $\widehat \Gr_\cP(v,\tilde q^w)$ and
$\Gr_\cP(v,\tilde q^w)$ given by
\begin{gather*}
  (g,h) \star \gamma = \gamma_{(g,h)} \gamma,\quad
  \gamma \in \widehat \Gr_\cP(v,\tilde q^w) \\
  (g,h) \star U = \gamma_{(g,h)}(U),\quad U \in
  \Gr_\cP(v,\tilde q^w).
\end{gather*}

\begin{Proposition} \label{prop:isoms-equivariance}
The isomorphisms of Theorem~\ref{thm:QG=QV} are $\GL_w \times
G_\cP$-equivariant.
\end{Proposition}

\begin{proof}
Let $(x,t) \mapsto \gamma(x,t)$ be the map $\Lambda(v,w)^\st
\xrightarrow{\cong} \widehat \Gr_\cP(v,\tilde q^w)$ of
Theorem~\ref{thm:QG=QV}. Fix $(x,t) \in \Lambda(v,w)^\st$. Recall
that for $(g,h) \in \GL_w \times G_\cP$, we have $(g,h) \star (x,t) =
(h \star x, g t)$. Let $V^x$ be the $\mathcal{P}$-module
corresponding to $x$. Then ${^h V}^x$ is the $\mathcal{P}$-module
corresponding to $h \star x$. We have the commutative diagram:
\[
  \xymatrix{
  & \tilde q^w \ar[d]^\pi \\
  V^x \ar[ur]^{\gamma(x,t)} \ar[r]_t & s^w
  }
\]
It follows that the diagram
\[
  \xymatrix{
  & {^h \tilde q}^w \ar[d]^{\pi}
  \ar[rr]^{\gamma_{(g, h)}}
  & & \tilde q^w \ar[d]^\pi \\
  {^h V}^x \ar[ur]^{\gamma(x,t)} \ar[r]_{t} &
  s^w \ar[rr]_g & & s^w
  }
\]
commutes.  By the uniqueness statement in
Proposition~\ref{prop:unique-injection}, we have
\[
  \gamma((g, h) \star (x,t)) = \gamma(h \star x,
  g t) = \gamma_{(g,h)} \gamma(x,t)
  =(g,h) \star \gamma(x,t),
\]
which proves that the map $\Lambda(v,w)^\st \cong \widehat
\Gr_\cP(v,\tilde q^w)$ is equivariant.  The remaining claim follows from
the fact that the isomorphism $\mathfrak{L}(v,w) \cong
\Gr_\cP(v,\tilde q^w)$ is obtained from the map $\Lambda(v,w)^\st
\xrightarrow{\cong} \widehat \Gr_\cP(v,\tilde q^w)$ by taking quotients by
$\GL_V$.
\end{proof}

\subsection{Graded/cyclic quiver grassmannians}

Fix an abelian reductive subgroup $A$ and a group homomorphism $\rho
: A \to \GL_w \times G_\cP$, defining an action of $A$ on $\tilde q^w$
by $\mathcal{P}_0$-module automorphisms.  The weight space
corresponding to $\lambda \in \Hom (A, \C^*)$ is
\begin{equation} \label{eq:qw-lambda-def}
  \tilde q^w(\lambda) \stackrel{\text{def}}{=} \{v \in \tilde q^w\ |\ \rho(a)(v)
  = \lambda(a) v \ \forall\ a \in A\}.
\end{equation}

We define
\[
  \Gr_\cP(\tilde q^w)^A = \{U \in \Gr_\cP(\tilde q^w) \ |\ \ \rho(a) \star U = U\
  \forall\ a \in A\},\quad \Gr_\cP(u,\tilde q^w)^A = \Gr_\cP(\tilde q^w)^A \cap
  \Gr_\cP(u,\tilde q^w).
\]
Then for all $U \in \Gr_\cP(\tilde q^w)^A$, we have the map $\rho_U : A \to
\GL_U$, $a \mapsto \rho(a)|_U$.  In other words,
$\rho_U$ is a representation of $A$ in the category of
$\mathcal{P}_0$-modules.  If $\rho_1$ and $\rho_2$ are two such
representations, we write $\rho_1 \cong \rho_2$ when $\rho_1$ and
$\rho_2$ are isomorphic.  That is, $\rho_1 \cong \rho_2$ for $\rho_i
: A \to \GL_{U_i}$, if there exists a
$\mathcal{P}_0$-module isomorphism $\xi : U_1 \to U_2$ such that
$\rho_2 = \xi \rho_1 \xi^{-1}$, where $\xi \rho_U \xi^{-1}$ denotes
the homomorphism $a \mapsto \xi \rho_U(a) \xi^{-1}$.  Then, for
$\rho_1 : A \to \GL_U$, $U$ a
$\mathcal{P}_0$-module, we define
\[
  \Gr_\cP(\rho_1,\tilde q^w)^A = \{U' \in \Gr_\cP(\tilde q^w)^A \ |\ \rho_{U'} \cong \rho_1 \}.
\]
Note that $\Gr_\cP(\rho_1,\tilde q^w)^A$ depends only on the isomorphism
class of $\rho_1$.

Recall the action of $\GL_w \times G_{\cP}$ on $\Lambda(V,W)^\st$ and
$\mathfrak{L}(v,w)$ described in Section~\ref{subsec:quiver-action} (where we now identify $W$ with $s^w$, $w = \dim_{Q_0} W$).
Define
\[
  \mathfrak{L}(w)^A = \{[x,t] \in \mathfrak{L}(v,w)\  | \ \rho(a) \star [x,t] = [x,t] \
  \forall\ a \in A\},\quad \mathfrak{L}(v,w)^A = \mathfrak{L}(w)^A \cap
  \mathfrak{L}(v,w).
\]
Fix a point $[x,t] \in \mathfrak{L}(v,w)^A$.  For every $a \in A$,
there exists a unique $\rho_1(a) \in \GL_V$ such that
\begin{equation} \label{eq:QV-action}
  \rho(a) \star (x,t) = \rho_1^{-1}(a) \cdot (x,t),
\end{equation}
and the map $\rho_1 : A \to \GL_V$ is a homomorphism.  We
let $\mathfrak{L}(\rho_1,w)^A \subseteq \mathfrak{L}(v,w)^A$ be the
set of $A$-fixed points $y$ such that \eqref{eq:QV-action} holds for
some representative $(x,t)$ of $y$.

\begin{Theorem} \label{thm:fixed-point-isom}
Let $V$ be a $\mathcal{P}_0$-module and $\rho_1 : A \to
\GL_V$ a group homomorphism. Then
$\Gr_\cP(\rho_1,\tilde q^w)^A$ is isomorphic to $\mathfrak{L}(\rho_1,w)^A$
as an algebraic variety.
\end{Theorem}

\begin{proof}
Choose $[x,t] \in \mathfrak{L}(\rho_1,w)^A$.  Let $U =
\gamma(x,t)(V)$ be the corresponding point of $\Gr_\cP(v,\tilde q^w)^A$. We
want to show that $\rho_1 \cong \rho_U$. Let $(g,h) \in A$ and
consider the following commutative diagram:
\[
  \xymatrix{
  & {^h \tilde q}^w \ar[d]^\pi \ar[rr]^{\gamma_{(g,h)}}
  & & \tilde q^w \ar[d]_\pi & \\
  {^h V}^x \ar[ur]^{\gamma(x,t)} \ar[r]_t &
  s^w \ar[rr]_g & & s^w
  & V\ . \ar[ul]_{\gamma(x,t)} \ar[l]^t
  }
\]
Then $\rho_U(g,h) = \gamma_{(g,h)}|_U$. Note that $\gamma(x,t)$ is
an isomorphism when its codomain is restricted to $U$ and we denote
by $\gamma(x,t)^{-1}$ the inverse of this restriction.  We claim
that $\rho_1 = \tilde \rho \stackrel{\text{def}}{=} \gamma(x,t)^{-1}
\left( \gamma_{(g,h)}|_U \right) \gamma(x,t)$.  It suffices to show
that
\[
  (h \star x, g t) = (g, h) \star (x,t)
  = {\tilde \rho}^{-1} \cdot (x,t) = ({\tilde \rho}^{-1} x {\tilde \rho}, t {\tilde \rho}).
\]
We have
\begin{align*}
  {\tilde \rho}^{-1} x &= \gamma(x,t)^{-1} (\gamma_{(g,h)}|_U)^{-1}
  \gamma(x,t) x \\
  &= \gamma(x,t)^{-1} (\gamma_{(g,h)}|_U)^{-1}
  x \gamma(x,t) \\
  &= \gamma(x,t)^{-1} (h \star x) (\gamma_{(g,z)}|_U)^{-1}
  \gamma(x,t) \\
  &= (h \star x) \gamma(x,t)^{-1} (\gamma_{(g,z)}|_U)^{-1}
  \gamma(x,t) \\
  &= (h \star x) {\tilde \rho}^{-1}
\end{align*}
and so ${\tilde \rho}^{-1} x {\tilde \rho} = h \star x$. Similarly,
$t {\tilde \rho} =  t \gamma(x,t)^{-1} \left( \gamma_{(g,h)}|_U
\right) \gamma(x,t) = g t$ and we are done.
\end{proof}

We now restrict to a special case of the above construction which
has been studied by Nakajima.  In particular, we define $\GL_w \times
\C^*$-actions on the quiver grassmannians corresponding to the
actions on quiver varieties described in
Section~\ref{subsec:quiver-action}.

For any function $m: \tilde Q_1 \to \Z$ such that $m(a) = -m(\bar a)$ for
all $a \in \tilde Q_1$, the group homomorphism \eqref{eq:GC-to-GP-homom}
defines a $\GL_w \times \C^*$-action on $\tilde q^w$, $\widehat
\Gr_\cP(v,\tilde q^w)$ and $\Gr_\cP(v,\tilde q^w)$ which we again denote by
$\star_m$.  If $A$ is any abelian reductive subgroup of $\GL_w \times
\C^*$, we can consider the weight decompositions as above. For the
remainder of this section, we fix $m=m_2$ (see
Section~\ref{subsec:quiver-action}). That is, $m(a)=0$ for all $a
\in Q_1$.  We also write $\star$ for $\star_m$.  Recall the definition \eqref{eq:qw-lambda-def} of $\tilde q^w(\lambda)$.  For $x \in
\mathcal{P}_n$, $v \in \tilde q^w(\lambda)$ and $(g,z) \in A$, we have
\[
  \rho(g,z)(x \cdot v) = \gamma_{(zg,h_m(z))}(x \cdot v) = z^{-n} x
  \cdot \gamma_{(zg,h_m(z))} (v) = z^{-n} \lambda(g,z) v.
\]
Thus $\mathcal{P}_n : \tilde q^w(\lambda) \to \tilde q^w(l^{-n} \lambda)$, where
we write $l^{-n}\lambda$ for the element $L(-n) \otimes \lambda$ of
$\Hom(A,\C^*)$ and $L(-n)=\C$ with $\C^*$-module structure given by
$z \cdot v = z^{-n} v$.

Now let $(g,z)$ be a semisimple element of $A$ and define
\[
  \Gr_\cP(\tilde q^w)^{(g,z)} = \{U \in \Gr_\cP(\tilde q^w) \ |\
  (g,z) \star U = U\},\quad \Gr_\cP(u,\tilde q^w)^{(g,z)} =
  \Gr_\cP(\tilde q^w)^{(g,z)} \cap \Gr_\cP(u,\tilde q^w).
\]
The module $\tilde q^w$ has an eigenspace decomposition with respect to the
action of $(g,z)$ given by
\[
  \tilde q^w = \bigoplus_{a \in \C^*} \tilde q^w(a),\quad \tilde q^w(a) = \{v \in \tilde q^w\ |\
  (g,z) \star v = av\}.
\]
Then $\Gr_\cP(\tilde q^w)^{(g,z)}$ consists of those $U \in \Gr_\cP(\tilde q^w)$
that are direct sums of subspaces of the weight spaces $\tilde q^w(a)$, $a
\in \C^*$.  Thus, each $U \in \Gr_\cP(\tilde q^w)^{(g,z)}$ inherits a
weight space decomposition, or $\C^*$-grading,
\[
  U = \bigoplus_{a \in \C^*} U(a),\quad U(a) = \{v \in U\ |\
  (g,z) \star v = av\}.
\]
As above we see that $\mathcal{P}_n : \tilde q^w(a) \to \tilde q^w(a z^{-n})$ and
$\mathcal{P}_n : U(a) \to U(a z^{-n})$.  We also regard $s^w$ as an
$A$-module via the composition
\[
  A \hookrightarrow \GL_w \times \C^*
  \xrightarrow{\text{projection}} \GL_w = \GL_{s^w}.
\]
Thus $s^w$ also inherits a $\C^*$-grading as above. For a $Q_0
\times \C^*$-graded vector space $V = \bigoplus_{i \in Q_0,\, a \in
\C^*} V_{i,a}$, define the graded dimension (or character)
\[
  \cha V = \sum_{i \in Q_0,\, a \in \C*} (\dim V_{i,a})X_{i,a} \in
  \N[X_{i,a}]_{i \in Q_0,\, a
\in \C^*}.
\]
Recall that a $\mathcal{P}_0$-module is equivalent to an
$Q_0$-graded vector space. Thus $\tilde q^w$, $s^w$, and elements of
$\Gr_\cP(\tilde q^w)^{(g,z)}$ have natural $Q_0 \times \C^*$-gradings and
we can consider their graded dimensions.

\begin{Definition}[Graded/cyclic quiver grassmannian]
For a graded dimension \\
$\mathbf{d} \in \N[X_{i,a}]_{i \in Q_0,\, a
\in \C^*}$, define
\[
  \Gr_\cP(\mathbf{d},\tilde q^w)^{(g,z)} = \{U \in
  \Gr_\cP(\tilde q^w)^{(g,z)}\ |\ \cha U = \mathbf{d} \}.
\]
We call $\Gr_\cP(\mathbf{d},\tilde q^w)^{(g,z)}$ a \emph{cyclic quiver grassmannian} if $z$ is a root of unity, and a \emph{graded quiver grassmannian} otherwise.
\end{Definition}

\begin{Theorem} \label{thm:graded-QG=QV-isom}
Let $V$ be a $Q_0 \times \C^*$-graded vector space. For a semisimple
element $(g, z) \in \GL_w \times \C^*$, the graded/cyclic quiver
grassmannian $\Gr_\cP(\cha V, \tilde q^w)^{(g,z)}$ is isomorphic to the
lagrangian graded/cylic quiver variety $\mathfrak{L}^\bullet(V,s^w)$
defined in \cite[\S4]{Nak04}, where $s^w$ is considered as a $Q_0
\times \C^*$-graded vector space as above.
\end{Theorem}

\begin{proof}
This follows immediately from
Proposition~\ref{prop:isoms-equivariance} since
$\mathfrak{L}^\bullet(V,W)$ is simply the set of points of
$\mathfrak{L}(V,W)$ fixed by a semisimple element $(g, z)$ of
$\GL_w \times \C^*$.
\end{proof}

\begin{Remark}
In \cite{Nak04}, Nakajima assumes the quiver $Q$ is of $ADE$ type.
However, the definitions in \cite[\S4]{Nak04} extend naturally to
the more general case.
\end{Remark}


\section{Geometric construction of representations of Kac-Moody algebras
and compatibility with nested quiver grassmannians}
\label{sec:geom-reps}

Since certain quiver grassmannians are isomorphic to lagrangian
Nakajima quiver varieties, one can translate Nakajima's geometric
construction of representations of Kac-Moody algebras into the
quiver grassmannian setting.  Having done this, one sees that the
quiver grassmannian construction is compatible with a natural
nesting of these varieties -- a property which seems to have no
analog in the setting of quiver varieties.  One benefit of this
nesting compatibility is that it allows one to always work with
quiver grassmannians in \emph{finite-dimensional} modules, even
though the injective objects $q^w$ and $\tilde q^w$ themselves may be
infinite-dimensional (outside of finite type).

For the remainder of this section, we fix a Kac-Moody algebra $\g$
with symmetric Cartan matrix and let $\mathcal{W}$ be its Weyl
group. Let $Q=(Q_0,Q_1)$ be a quiver whose underlying graph is the
Dynkin graph of $\g$ and let $\mathcal{P}=\cP(Q)$ denote the
corresponding path algebra. We also fix a $\mathcal{P}_0$-module
retract $\pi : \tilde q^w \to s^w$, allowing us to identify
$\Gr_\cP(v,\tilde q^w)$ with $\mathfrak{L}(v,w)$ as in Remark \ref{rem:tqg}.

\subsection{Constructible functions}

Recall that for a topological space $X$, a \emph{constructible set}
is a subset of $X$ that is obtained from open sets by a finite
number of the usual set theoretic operations (complement, union and
intersection).  A \emph{constructible function} on $X$ is a function
that is a finite linear combination of characteristic functions of
constructible sets.  For a complex variety $X$, let $M(X)$ denote
the $\C$-vector space of constructible functions on $X$ with values
in $\C$. We define $M(\emptyset)=0$. For a continuous map $p : X \to
X'$, define
\begin{gather*}
    p^* : M(X') \to M(X),\quad (p^*f')(x) = f'(p(x)),\quad f' \in
    M(X'),\\
    p_! : M(X) \to M(X'),\quad (p_!f)(x) = \sum_{a \in \Q} a \chi
    (p^{-1}(x) \cap f^{-1}(a)),\quad f \in M(X),
\end{gather*}
where $\chi$ denotes the Euler characteristic of cohomology with
compact support.

\begin{Lemma}
Suppose $X$ is a constructible subset of a topological space $Y$ and
let $\iota : X \hookrightarrow Y$ be the inclusion map. Then
\begin{enumerate}
  \item $\iota^*(f) = f|_X$ for $f \in M(Y)$, and
  \item for $f \in M(X)$, $\iota_!(f)$ is the extension of $f$ by zero.
  That is $\iota_!(f)(x) = f(x)$ for $x \in X$ and $\iota_!(f)(x) =
  0$ for $x \in Y \setminus X$.
\end{enumerate}
\end{Lemma}

\begin{proof}
The proof is straightforward and will be omitted.
\end{proof}

\subsection{Raising and lowering operators}

Let $V$ be a locally nilpotent $\mathcal{P}$-module with finite-dimensional socle. By Proposition~\ref{prop:unique-injection} and Lemma~\ref{lem:loc-nil-contained-in-tqw}, $V$ is isomorphic to a submodule of $\tilde q^w$ for some $w$, and in particular, for any $u$, $\Gr(u,V)$ is naturally a variety.
For $u, u' \in \N Q_0$ with $u
\le u'$ (i.e.\ $u=\sum u_i i$ and $u' = \sum u_i' i$ where $u_i \le
u_i'$ for all $i \in Q_0$), define
\begin{equation} \label{eq:intermediate-variety}
  \Gr_\cP(u,u',V) = \{(U, U') \in \Gr_\cP(u,V) \times
  \Gr_\cP(u',V) \ |\ U \subseteq U'\},
\end{equation}
and let
\[
  \Gr_\cP(u,V) \xleftarrow{\pi_1} \Gr_\cP(u,u',V)
  \xrightarrow{\pi_2} \Gr_\cP(u',V)
\]
be the natural projections given by $\pi_1(U,U')=U$ and $\pi_2(U,U')
= U'$.  For each $i \in I$, define the following operators:
\begin{equation} \label{def:raise-lower}
  \begin{array}{rcll}
    \hat E_i &:& M(\Gr_\cP(u+i,V)) \to M(\Gr_\cP(u,V)), &
    \hat E_i f = (\pi_1)_! (\pi_2^* f),\\
    \hat F_i &:& M(\Gr_\cP(u,V)) \to M(\Gr_\cP(u+i,V)), &
    \hat F_i f = (\pi_2)_! (\pi_1^* f).
  \end{array}
\end{equation}
where the maps $\pi_1$ and $\pi_2$ are as in
\eqref{eq:intermediate-variety} with $u'=u+i$.

\subsection{Compatibility with nested quiver grassmannians}

Suppose $V_1 \subseteq V_2$ are $\cP$-modules.  Then we have the
commutative diagram
\[
  \xymatrix{
    \Gr_\cP(u,V_1) \ar@{^{(}->}[d]^{\iota_u} & \Gr_\cP(u,u',V_1)
    \ar[l]_{\pi_1^1} \ar[r]^{\pi_2^1}  \ar@{^{(}->}[d]_{\iota_{u,u'}}
    & \Gr_\cP(u',V_1)
    \ar@{^{(}->}[d]_{\iota_{u'}} \\
    \Gr_\cP(u,V_2) & \Gr_\cP(u,u',V_2) \ar[l]_{\pi_1^2}
    \ar[r]^{\pi_2^2}  & \Gr_\cP(u',V_2)
  }
\]
where $\iota_u$, $\iota_{u'}$ and $\iota_{u,u'}$ denote the
canonical inclusions.  Denote by $\hat E_i^j$ and $\hat F_i^j$,
$j=1,2$, the operators defined in \eqref{def:raise-lower} for
$V=V_j$.

\begin{Proposition} \label{prop:restricted-operators}
We have
\begin{enumerate}
  \item $\hat E_i^1 = \iota_u^* \circ \hat E_i^2 \circ
  (\iota_{u+i})_!$, and
  \item $\hat F_i^1 = \iota_{u+i}^* \circ \hat F_i^2 \circ
  (\iota_u)_!$.
\end{enumerate}
\end{Proposition}

\begin{proof}
Let $u'=u+i$.  By linearity, it suffices to prove the first
statement for functions of the form $1_X$ where $X$ is a
constructible subset of $\Gr_\cP(u',V_1)$.  Then $(\iota_{u'})_! 1_X
= 1_X$, where on the righthand side, $X$ is viewed as a subset of
$\Gr_\cP(u',V_2)$. We have
\[
  (\pi_2^2)^* \circ (\iota_{u'})_! 1_X = (\pi_2^2)^* 1_X =
  1_{(\pi_2^2)^{-1}(X)},
\]
and
\[
  (\iota_{u,u'})_! (\pi_2^1)^* 1_X = (\iota_{u,u'})_!
  1_{(\pi_2^1)^{-1}(X)} = 1_{(\pi_2^1)^{-1}(X)}.
\]
Since $X \subseteq \Gr_\cP(u',V_1)$, we have $(\pi_2^2)^{-1}(X) =
(\pi_2^1)^{-1}(X)$ and thus
\[
  (\pi_2^2)^* \circ (\iota_{u'})_! 1_X = (\iota_{u,u'})_! \circ (\pi_2^1)^*
  1_X.
\]
Therefore
\begin{align*}
  \iota_u^* \circ \hat E_i^2 \circ (\iota_{u'})_! 1_X &= \iota_u^*
  \circ (\pi_1^2)_! \circ (\pi_2^2)^* \circ (\iota_{u'})_! 1_X \\
  &= \iota_u^* \circ (\pi_1^2)_! \circ (\iota_{u,u'})_! \circ (\pi_2^1)^*
  1_X \\
  &= \iota_u^* \circ (\pi_1^2 \circ \iota_{u,u'})_! \circ (\pi_2^1)^* 1_X
  \\
  &= \iota_u^* \circ (\iota_u \circ \pi_1^1)_! \circ (\pi_2^1)^* 1_X \\
  &= \iota_u^* \circ (\iota_u)_! \circ (\pi_1^1)_! \circ (\pi_2^1)^*
  1_X \\
  &= (\pi_1^1)_! \circ (\pi_2^1)^* 1_X \\
  &= \hat E_i^1 1_X,
\end{align*}
where the sixth equality holds since $\iota_u^* \circ (\iota_u)_!$
is the identity on $M(\Gr_\cP(u,V_1))$.

We now prove the second statement.  Again, it suffices to prove it
for functions of the form $1_X$ where $X$ is a constructible subset
of $\Gr_\cP(u,V_1)$. Now, for $U \in \Gr_\cP(u',V_1)$, we have
\begin{align*}
  \iota_{u'}^* \circ \hat F_i^2 \circ (\iota_u)_! 1_X(U) &=
  \iota_{u'}^* \circ (\pi_2^2)_! \circ (\pi_1^2)^* \circ
  (\iota_u)_! 1_X(U) \\
  &= \iota_{u'}^* \circ (\pi_2^2)_! \circ (\pi_1^2)^* 1_X(U) \\
  &= \iota_{u'}^* \circ (\pi_2^2)_! \circ 1_{(\pi_1^2)^{-1}(X)} (U)
  \\
  &= \chi \left( (\pi_2^2)^{-1}(U) \cap (\pi_1^2)^{-1}(X) \right) \\
  &= \chi \left( (\pi_2^1)^{-1}(U) \cap (\pi_1^1)^{-1}(X) \right) \\
  &= (\pi_2^1)_! 1_{(\pi_1^1)^{-1}(X)}(U) \\
  &= (\pi_2^1)_! \circ (\pi_1^1)^* 1_X (U) \\
  &= \hat F_i^1 1_X(U),
\end{align*}
where the fifth equality holds since $U \in \Gr_\cP(u',V_1)$.
\end{proof}

It follows from Proposition~\ref{prop:DQG=DQV} that the Demazure
quiver grassmannians stabilize in the following sense.

\begin{Corollary} \label{cor:stabilize}
For $u,w \in \N Q_0$, there exists $\sigma \in \mathcal{W}$, such
that $\Gr_\cP(v,q^{w,\sigma'})$ is isomorphic to $\mathfrak{L}(v,w)$
for all $\sigma' \succeq \sigma$.
\end{Corollary}

\begin{proof}
It follows from \cite[Proposition~6.1]{Sav04a} that there exists a
$\sigma \in \mathcal{W}$ such that $\Gr_\cP(v, q^{w,\sigma}) \cong
\mathfrak{L}_\sigma(v,w) = \mathfrak{L}(v,w)$.  It follows from the
same proposition that for $\sigma' \succeq \sigma$, we have
$\mathfrak{L}_{\sigma'}(v,w) = \mathfrak{L}(v,w)$.  The result then
follows from Proposition~\ref{prop:DQG=DQV}.
\end{proof}

\begin{Corollary} \label{cor:DQG-stabilization}
For $v, w \in \N Q_0$, let $\sigma^{v,w} \in \mathcal{W}$ be minimal
among the $\sigma \in \mathcal{W}$ such that
$\Gr_\cP(v,q^{w,\sigma})$ is isomorphic to $\mathfrak{L}(v,w)$. Then
$\Gr_\cP(v,q^{w,\sigma}) \cong \Gr_\cP(v,\tilde q^w)$ for all $\sigma
\succeq \sigma^{v,w}$.  In particular, every nilpotent submodule of the
injective module $q^w$ of graded dimension $v$ is a submodule of
$q^{w,\sigma}$ for $\sigma \succeq \sigma^{v,w}$.
\end{Corollary}

\begin{Remark}
In the case when $\g$ is of finite type, we can take $\sigma =
\sigma_0$, where $\sigma_0$ is the longest element of the Weyl
group. Then $\Gr_\cP(v,q^w)$ is isomorphic to
$\Gr_\cP(v,q^{w,\sigma_0})$ for all $v \in \N Q_0$.
\end{Remark}

\begin{Lemma} \label{lem:projection-cd}
Suppose $w,v,v' \in \N Q_0$ with $v \le v'$ and $\sigma \in
\mathcal{W}$. Then the diagram
\[
  \xymatrix{
    \Gr_\cP(v,q^{w,\sigma}) \ar@{^{(}->}[d] &
    \Gr_\cP(v,v',q^{w,\sigma}) \ar[r]^{\pi_2} \ar[l]_{\pi_1} \ar@{^{(}->}[d] &
    \Gr_\cP(v',q^{w,\sigma}) \ar@{^{(}->}[d] \\
    \Gr_\cP(v,\tilde q^w) &
    \Gr_\cP(v,v',\tilde q^w) \ar[r]^{\pi_2} \ar[l]_{\pi_1} &
    \Gr_\cP(v',\tilde q^w)
  }
\]
commutes, where the vertical arrows are the natural inclusions.  If
$\sigma \succeq \sigma^{v,w}, \sigma^{v',w}$, then the vertical
arrow are isomorphisms.
\end{Lemma}

\begin{proof}
This follows immediately from Corollary~\ref{cor:DQG-stabilization}.
\end{proof}

\subsection{Quiver grassmannian realization of representations}

For each $i \in I$, define
\begin{equation} \label{def:Hi}
  \begin{array}{rcll}
    H_i : M(\Gr_\cP(v,\tilde q^w)) \to M(\Gr_\cP(v,\tilde q^w)),\quad H_i f = (w-Cv)_i f,
  \end{array}
\end{equation}
where $C$ is the Cartan matrix of $\g$.  Also, in the special case
when $V=\tilde q^w$ for some $w$, we denote the operators $\hat E_i$ and
$\hat F_i$ by $E_i$ and $F_i$ respectively.

\begin{Proposition}
The operators $E_i$, $F_i$, $H_i$ define an action of $\g$ on
$\bigoplus_u M(\Gr_\cP(u,\tilde q^w))$.
\end{Proposition}

\begin{proof}
Throughout this proof, for varieties $X$ and $Y$, the notation $X
\cong Y$ means that $X$ and $Y$ are homeomorphic.  In
\cite[\S10]{Nak94}, Nakajima defines the variety
\[
    \mathfrak{F}(v,w;i) \stackrel{\text{def}}{=} {\tilde
    {\mathfrak{F}}}(v,w;i)/\GL_V,
\]
where
\[
    {\tilde {\mathfrak{F}}}(v,w;i) = \{(x,t,Z)\ |\ (x,t) \in
    \Lambda(V,W)^\st,\ Z \subseteq V,\ x(Z) \subseteq Z,\ \dim
    Z = v-i\}.
\]
Using the homeomorphism of Theorem~\ref{thm:QG=QV}, we have
\[
    \tilde {\mathfrak{F}}(v,w;i) \cong \{(\gamma,Z)\ |\ \gamma \in
    \widehat \Gr_\cP(v,\tilde q^w),\ Z
    \subseteq V,\ \dim Z = v-i,\ \cP \cdot \gamma(Z)) \subseteq \gamma(Z)\}.
\]
The map
\begin{gather*}
    \{(\gamma,Z)\ |\ \gamma \in \widehat \Gr_\cP(v,\tilde q^w),\ Z \subseteq V,\
    \dim Z = v-i,\ \cP \cdot \gamma(Z) \subseteq
    \gamma(Z)\} \to \Gr_\cP(v-i,v,\tilde q^w),\\
    (\gamma,Z) \mapsto (\gamma(Z), \gamma(V))
\end{gather*}
is a principal $\GL_V$-bundle and thus
\begin{align*}
    \mathfrak{F}(v,w;i) &= {\tilde {\mathfrak{F}}}(v,w;i)/\GL_V
    \\
    &\cong \{(\gamma,Z)\ |\ \gamma \in \widehat \Gr_\cP(v,\tilde q^w),\
    Z \subseteq V,\ \dim Z = v-i,\ \cP \cdot \gamma(Z) \subseteq
    \gamma(Z)\}/ \GL_V \\
    &= \Gr_\cP(u-i,u,\tilde q^w).
\end{align*}
Therefore, the following diagram commutes:
\begin{equation} \label{eq:QV-QG-commutative-diagram}
    \xymatrix{ \Gr_\cP(v-i,\tilde q^w) \ar[d]^{\cong} & \Gr_\cP(v-i,v,\tilde q^w)
    \ar[l]_{\pi_1}
    \ar[r]^{\pi_2} \ar[d]^{\cong} & \Gr_\cP(v,\tilde q^w) \ar[d]^{\cong} \\
    \mathfrak{L}(v-i,w) & \mathfrak{F}(v,w;i)
    \ar[r]^(.5){\pi_2}
    \ar[l]_(.4){\pi_1} & \mathfrak{L}(v,w) }
\end{equation}
where the maps $\pi_1$ and $\pi_2$ appearing on the bottom row are
described in \cite[\S10]{Nak94}.  The result then follows
immediately from \cite[Proposition~10.12]{Nak94}.
\end{proof}

Let $U(\g)^-$ be the lower half of the enveloping algebra of $\g$.
Then let $\alpha$ be the constant function on $\Gr_\cP(0,\tilde q^w)$ with
value 1 and let
\begin{align}
  L_w &\stackrel{\text{def}}{=} U(\g)^- \cdot \alpha \subseteq
    \bigoplus_v M(\Gr_\cP(v,\tilde q^w)),\\
    L_w(v) &\stackrel{\text{def}}{=} M(\Gr_\cP(v,\tilde q^w)) \cap L_w.
\end{align}

\begin{Theorem} \label{thm:geom-rep-Demazure}
The operators $E_i$, $F_i$, $H_i$ preserve $L_w$ and $L_w$ is
isomorphic to the irreducible highest-weight integrable
representation of $\g$ with highest weight $\omega_w$.  The summand
$L_w(v)$ in the decomposition $L_w = \bigoplus_v L_w(v)$ is a weight
space with weight $\omega_w - \alpha_v$.
\end{Theorem}

\begin{proof}
In light of the commutative diagram
\eqref{eq:QV-QG-commutative-diagram}, the result follows immediately
from \cite[Theorem~10.14]{Nak94}.
\end{proof}

\begin{Remark}
By Proposition~\ref{prop:restricted-operators} and
Lemma~\ref{lem:projection-cd}, we can always work with
$\Gr_\cP(v,q^{w,\sigma})$ for large enough $\sigma$. Therefore, we
can avoid quiver grassmannians in infinite-dimensional injectives if
desired.
\end{Remark}

From the realization of irreducible highest-weight representations
given in Theorem~\ref{thm:geom-rep-Demazure}, we obtain some natural
automorphisms of these representations.  Recall from
Definition~\ref{def:Aut-action-on-QG} the natural action of
$\Aut_\cP \tilde q^w$ on $\Gr_\cP(v,\tilde q^w)$ for any $v$ given by $(g, V)
\mapsto g(V)$. This induces an action on $\bigoplus_v
M(\Gr_\cP(v,\tilde q^w))$ given by
\[
  (g, f) \mapsto f \circ g^{-1},\quad f \in
  \bigoplus_v M(\Gr_\cP(v,\tilde q^w)),\quad g \in \Aut_\cP q_w.
\]
This action clearly commutes with the operators $E_i$ and $F_i$ and
thus induces an action on $L_w$.  Such actions do not seem to be
clear in the original quiver variety picture.  Similar actions have
been considered by Lusztig \cite[\S1.22]{Lus00b} in the case when
$Q$ is of finite type.


\section{Relation to Lusztig's grassmannian realization}
\label{sec:injective-versus-projective}

In \cite{Lus98,Lus00b}, Lusztig gave a grassmannian type realization
of the lagrangian Nakajima quiver varieties inside the projective
modules $p^w$.  In the case when $Q$ is a quiver of finite type, the
injective hulls of the simple objects are also projective covers (of
different simple objects).  Thus, Lusztig's and our construction are
closely related.  In this section, we extend Lusztig's construction
to give a realization of the Demazure quiver varieties.  We then
give a precise relationship between his construction and ours in the
finite type case. We will see that the natural identification of the
two constructions corresponds to the Chevalley involution on the
level of representations of the Lie algebra $\g$ associated to our
quiver.

\subsection{Lusztig's construction and Demazure quiver varieties}

\begin{Definition} \label{def:PQG}
  For $V \in \cP$-Mod, define
  \[
    \tilde \Gr_\cP(V) = \{U \in \Gr_\cP(V) \ |\ \cP_n \cdot V
    \subseteq U \text{ for some } n \in \N\}.
  \]
  In other words, $\tilde \Gr_\cP(V)$ consists of all
  $\cP$-submodules of
  $V$ such that the quotient $V/U$ is nilpotent.  For $u \in \N Q_0$, we define
  \[
    \tilde \Gr_\cP(u,V) = \{U \in \tilde \Gr_\cP(V)\ |\ \dim_{Q_0} (V/U) =
    u\}.
  \]
\end{Definition}

\begin{Proposition} \label{prop:QGproj=QV}
  Fix $v,w \in \N Q_0$. Then $\mathfrak{L}(v,w)$ is isomorphic to
  $\tilde \Gr_\cP(v,p^w)$ as an algebraic variety.
\end{Proposition}

\begin{proof}
  This is proven in \cite[Corollary~3.2]{Shi10}.  Note that, in \cite{Shi10}, a different stability condition is used in the definition of $\mathfrak{L}(v,w)$.  However, it is well-known that the different stability conditions give rise to isomorphic varieties.  We refer the reader to \cite{Nak96} for a discussion of various stability conditions.
\end{proof}

\begin{Proposition} \label{prop:proj-extremal}
  For $v \in \N Q_0$, the following statements are equivalent:
  \begin{enumerate}
    \item \label{prop-item:proj-extremal1} $v$ is $w$-extremal,
    \item \label{prop-item:proj-extremal2} $\mathfrak{L}(v,w)$ consists of a single point,
    \item \label{prop-item:proj-extremal3} $\tilde \Gr_\cP(v,p^w)$ consists of a single point, and
    \item \label{prop-item:proj-extremal4} there is a unique $\cP$-submodule $V$ of $p^w$
    of codimension $v$ such that $p^w/V$ is nilpotent.
  \end{enumerate}
\end{Proposition}

\begin{proof}
The equivalence of \eqref{prop-item:proj-extremal1} and
\eqref{prop-item:proj-extremal2} is given in
\cite[Proposition~5.1]{Sav04a}. The equivalence of
\eqref{prop-item:proj-extremal2} and
\eqref{prop-item:proj-extremal3} follows from
Proposition~\ref{prop:QGproj=QV}.  Finally, the equivalence of
\eqref{prop-item:proj-extremal3} and
\eqref{prop-item:proj-extremal4} follows directly from
Definition~\ref{def:PQG}
\end{proof}

\begin{Definition} \label{def:PDQG}
  For $\sigma \in \mathcal{W}$, we let $p^{w,\sigma}$ denote the
  unique submodule of $p^w$ of graded codimension $\sigma \cdot_w
  0$ and define
  \[
    \tilde \Gr_{Q,\sigma}(v,p^w) = \{V \in \tilde \Gr_\cP(v,p^w)\ |\
    p^{w,\sigma} \subseteq V\}.
  \]
\end{Definition}

\begin{Proposition} \label{prop:PDQG=DWV}
  Fix $\sigma \in \mathcal{W}$ and $v, w \in \N Q_0$.  Then
  $\tilde \Gr_{Q,\sigma}(v,p^w)$ is isomorphic to the Demazure quiver variety
  $\mathfrak{L}_{\sigma}(v,w)$.
\end{Proposition}

\begin{proof}
  This follows immediately from Definitions~\ref{def:DQV}
  and~\ref{def:PDQG}, and Proposition~\ref{prop:QGproj=QV}.
\end{proof}

\subsection{Relation between the projective and injective
constructions}

We now suppose $Q$ is of finite type and let $\g$ be the Kac-Moody
algebra whose Dynkin diagram is the underlying graph of $Q$. Let
$\sigma_0$ be the longest element of the Weyl group of $\g$. There
is a unique Dynkin diagram automorphism $\theta$ such that
$-w_0(\alpha_i) = \alpha_{\theta(i)}$.  Extend $\theta$ to an
automorphism of the root lattice $\bigoplus_{i \in Q_0} \Z \alpha_i$
by linearly extending the map $\alpha_i \mapsto \alpha_{\theta(i)}$.
We also have an involution of $\N Q_0$ given by $w \mapsto
\theta(w)$ where $\theta(w)_i = w_{\theta(i)}$.

\begin{Definition}[Chevalley involution]
The \emph{Chevalley involution} $\zeta$ of $\g$ is given by
\[
  \zeta(E_i) = F_i,\quad \zeta(F_i) = E_i,\quad \zeta(H_i) = -H_i.
\]
For any representation $V$ of $\g$, let $^\zeta V$ be the
representation with the same underlying vector space as $V$, but
with the action of $\g$ twisted by $\zeta$.  More precisely, the
$\g$-action on $^\zeta V$ is given by $(a,v) \mapsto \zeta(a) \cdot
v$.
\end{Definition}

For a dominant weight $\lambda$ of $\g$, let $L_\lambda$ denote the
corresponding irreducible highest-weight representation and let
$v_\lambda$ be a highest weight vector.  Recall that an isomorphism
of irreducible representations is uniquely determined by the image
of $v_\lambda$.  The following lemma is well known.

\begin{Lemma} \label{lem:three-isoms} The lowest weight of $L_\lambda$ is $\sigma_0(\lambda)
= - \theta(\lambda)$.  If $v_{-\theta(\lambda)}$ denotes a lowest
weight vector, then the map $v_\lambda \mapsto v_{-\theta(\lambda)}$
induces an isomorphism $^\zeta L_{\lambda} \cong
L_{\theta(\lambda)}$.
\end{Lemma}

\begin{Lemma} \label{lem:dim-p-q}
  We have $\dim_{Q_0} p^w = \dim_{Q_0} q^w = \sigma_0 \cdot_w 0$.
\end{Lemma}

\begin{proof}
  Since the lowest weight of the
  representation $L(w)$ is $\sigma_0(w)$, the result follows
  immediately from Theorem~\ref{thm:QG=QV}
  and Proposition~\ref{prop:QGproj=QV}.
\end{proof}

\begin{Lemma} \label{lem:max-v-dim}
  For $w \in \N Q_0$, we have $\sigma_0 \cdot_w 0 = \sigma_0
  \cdot_{\theta(w)} 0$.  Furthermore, $\theta(\sigma_0 \cdot_w 0) =
  \sigma_0 \cdot_w 0$.
\end{Lemma}

\begin{proof}
  Let $v = \sigma_0
  \cdot_w 0$.  Then $\alpha_v = \omega_w -
  \sigma_0 (\omega_w) = \omega_w + \theta(\omega_w)$ and the results
  follow easily from the fact that $\theta^2 = \Id$.
\end{proof}

\begin{Proposition} \label{prop:injective=projective}
  If $Q$ is a quiver of finite type and $w \in \N Q_0$,
  then $p^w \cong q^{\theta(w)}$.
\end{Proposition}

\begin{proof}
  Since $p^w = \bigoplus_{i \in Q_0} (p^i)^{\oplus w_i}$
  and $q^w = \bigoplus_{i \in Q_0} (q^i)^{\oplus w_i}$,
  it suffices to prove the result for $w$ equal to
  $i$ for arbitrary $i \in Q_0$.

  Let $v = \sigma_0 \cdot_w 0 = \dim_{Q_0} p^i$.
  In the geometric realization of crystals via quiver varieties
  \cite{S02}, the point $\tilde \Gr_\cP(v,p^w) \cong \mathfrak{L}(v,w)$
  corresponds to the lowest weight element of the crystal
  $B_{\omega_i}$.  The lowest weight of the representation $L_{\omega_i}$
  is $\sigma_0 (\omega_i) = -\omega_{\theta(i)}$. Therefore, it
  follows from the geometric description of the crystals that
  $\dim_{Q_0} \socle p^i = \theta(i)$. By Lemmas~\ref{lem:dim-p-q}
  and~\ref{lem:max-v-dim}, we have
  \[
    \dim_{Q_0} p^i = \sigma_0 \cdot_w 0 = \sigma_0 \cdot_{\theta(w)} 0
    = \dim_{Q_0} q^{\theta(i)}.
  \]
  Thus, by Proposition~\ref{prop:extremal}, we have $p^i \cong
  q^{\theta(i)}$.
\end{proof}

\begin{Corollary} \label{cor:demazure-projective=injective}
  Suppose $Q$ is a quiver of finite type, $w \in \N Q_0$,
  and $\sigma \in \mathcal{W}$.  Then $q^{w,\sigma} \cong p^{\theta(w),
  \sigma \sigma_0}$.
\end{Corollary}

\begin{proof}
  Let $\tau = \sigma \sigma_0$ (and so $\sigma = \tau \sigma_0$).  In light of
  Propositions~\ref{prop:extremal},~\ref{prop:proj-extremal} and~\ref{prop:injective=projective}
  and Definitions~\ref{def:DQG} and~\ref{def:PDQG}, it suffices to prove
  that the codimension of $q^{w,\sigma}$ in $q^w$ is $\tau
  \cdot_{\theta(w)} 0$.

  Let $y = \tau \cdot_{\theta(w)} 0$, so that $\tau (\theta(w)) = \theta(w) - \alpha_y$, that is,
  \[
    \alpha_y = \theta(w) - \tau (\theta(w)).
  \]
  Now, let
  \begin{gather*}
    v = \dim_{Q_0} q^w = \sigma_0 \cdot_w 0 \implies \sigma_0(w) = w
    - \alpha_v, \\
    u = \dim_{Q_0} q^{w,\sigma} = \sigma \cdot_w 0 \implies
    \sigma(w) = w - \alpha_u.
  \end{gather*}
  Therefore
  \begin{align*}
    \sum_{i \in Q_0} (v_i - u_i) \alpha_i &= - \sigma_0(w) + \sigma(w)
    \\
    &= \theta(w) + \tau \sigma_0 (w) \\
    &= \theta(w) - \tau (\theta(w)),
  \end{align*}
  and so $y=v-u$ as desired.
\end{proof}

\begin{Proposition} \label{prop:S=tS}
  If $Q$ is a quiver of finite type, then $\Gr_\cP(u,q^w) \cong
  \tilde \Gr_\cP((\sigma_0 \cdot_w 0) - u, p^{\theta(w)})$.
\end{Proposition}

\begin{proof}
  Let $(x,V)$ be the quiver representation corresponding to the
  $\mathcal{P}$-module $q^w$ and let $v = \dim_{Q_0} V = \sigma_0 \cdot_w 0$.
  By Proposition~\ref{prop:injective=projective}, $(x,V)$ also
  corresponds to the $\mathcal{P}$-module $p^{\theta(w)}$.
  By Remark~\ref{rem:p-nilpotent}, $\cP_n \cdot p^w = 0$ for sufficiently
  large $n$.  Therefore
  \begin{align*}
    \Gr_\cP(u,q^w) &= \{ U \subseteq V\ |\ x(U) \subseteq U,\ \dim
    U = u\} \\
    &= \{ U \subseteq V\ |\ x(U) \subseteq U,\ \dim_{Q_0} V/U = v-u\}
    \\
    &\cong \tilde \Gr_\cP \left(v-u,p^{\theta(w)}\right). \qedhere
  \end{align*}
\end{proof}

By Proposition~\ref{prop:S=tS}, we have
\begin{equation} \label{eq:Chevalley-isomorphisms}
  \xymatrix{
  \mathfrak{L}(u,w) & & \ar[ll]^(.7){\cong}_(.7){\phi_w(u)} \Gr_\cP(u,q^w)
  \cong
  \tilde \Gr_\cP((\sigma_0 \cdot_w 0) - u, p^{\theta(w)})
  \ar[rrr]_(.58){\cong}^(.58){\psi_{\theta(w)}((\sigma_0 \cdot_w 0)-u)}
  & & &
  \mathfrak{L}((\sigma_0 \cdot_w 0)-u,\theta(w))
  }
\end{equation}
where $\phi_w(u)$ is the isomorphism of Theorem~\ref{thm:QG=QV} (see
Corollary~\ref{cor:QG=QV-variety-map}), and $\psi_{\theta(w)}(u)$ is
the isomorphism of Proposition~\ref{prop:QGproj=QV}.  Define
\[
  \phi_w = (\phi_w(u))_u : \Gr_\cP(q^w) \to \bigsqcup_u
  \mathfrak{L}(u,w),\quad \psi_w = (\psi_w(u))_u : \tilde
  \Gr_\cP(p^w) \to \bigsqcup_u \mathfrak{L}(u,w).
\]

\begin{Theorem}
  The isomorphism $\psi_{\theta(w)} \circ \phi_w^{-1}$ induces the
  involution $\zeta$.  More precisely, we have $a \circ (\psi_{\theta(w)}
  \circ \phi_w^{-1})^* = (\psi_{\theta(w)}
  \circ \phi_w^{-1})^* \circ
  \zeta(a)$, $a \in \g$, as operators on $L_w$, where $(\psi_{\theta(w)} \circ
  \phi_w^{-1})^*$ denotes the pullback of functions along $\psi_{\theta(w)} \circ
  \phi_w^{-1}$.
\end{Theorem}

\begin{proof}
For $u,u' \in \N Q_0$, define
\[
  \tilde \Gr_\cP(u,u',p^{\theta(w)}) = \{(U,U') \in \tilde \Gr_\cP (u,p^{\theta(w)})
  \times \tilde \Gr_\cP (u',p^{\theta(w)})\ |\ U' \subseteq U\}.
\]
The map $\psi_{\theta(w)}$ induces a isomorphism
\[
  \tilde \Gr_\cP(u,u',p^{\theta(w)}) \xrightarrow{\cong}
  \mathfrak{F}(u,\theta(w);u-u')
\]
for all $u, u' \in \N Q_0$ and we will also denote this collection
of isomorphisms by $\psi_{\theta(w)}$. Then we have the following
commutative diagram.
\[
  \mbox{} \hspace{-1.5cm}
  \xymatrix{
    \mathfrak{L}(u-i,w) & \mathfrak{F}(u,w;i) \ar[l]_{\pi_1}
    \ar[r]^{\pi_2} &
    \mathfrak{L}(u,w) \\
    \Gr_\cP(u-i,q^w) \ar[d]^{\cong} \ar[u]^{\phi_w}_{\cong} & \Gr_\cP(u-i,u,q^w)
    \ar[l]_{\pi_1}
    \ar[r]^{\pi_2} \ar[d]^{\cong} \ar[u]^{\phi_w}_{\cong} & \Gr_\cP(u,q^w)
    \ar[d]^{\cong} \ar[u]^{\phi_w}_{\cong} \\
    \tilde \Gr_\cP((\sigma_0 \cdot_w 0)-(u-i), p^{\theta(w)})
    \ar[d]_{\psi_{\theta(w)}}^{\cong} &
    \tilde \Gr_\cP((\sigma_0 \cdot_w 0)-u, (\sigma_0 \cdot_w 0)-(u-i), p^{\theta(w)})
    \ar[l]_<(0.1){\pi_2} \ar[r]^<(0.15){\pi_1} \ar[d]_{\psi_{\theta(w)}}^{\cong}
    & \tilde \Gr_\cP((\sigma_0 \cdot_w 0) - u, p^{\theta(w)})
    \ar[d]_{\psi_{\theta(w)}}^{\cong} \\
    \mathfrak{L}((\sigma_0 \cdot_w 0)-(u-i),\theta(w)) &
    \mathfrak{F}((\sigma_0 \cdot_w 0)-u,\theta(w);i) \ar[l]_{\pi_2} \ar[r]^{\pi_1} &
    \mathfrak{L}((\sigma_0 \cdot_w 0)-u,\theta(w))
  }
\]
It follows that for $f \in \bigoplus_u M(\mathfrak{L}(u,w))$, we
have
\begin{gather*}
  E_i \circ (\psi_{\theta(w)}
  \circ \phi_w^{-1})^* (f) = (\psi_{\theta(w)}
  \circ \phi_w^{-1})^* \circ F_i (f),\\
  F_i \circ (\psi_{\theta(w)} \circ \phi_w^{-1})^* (f) =
  (\psi_{\theta(w)}
  \circ \phi_w^{-1})^* \circ E_i (f).
\end{gather*}
Furthermore, $(\psi_{\theta(w)} \circ \phi_w^{-1})^*$ maps the
constant function on $\mathfrak{L}(0,w)$ with value one to the
constant function on $\mathfrak{L}(\sigma_0 \cdot_w 0,\theta(w))$
with value one. The result follows.
\end{proof}

\begin{Remark}
  Note that the middle isomorphism in \eqref{eq:Chevalley-isomorphisms}
  depends on our identification of $q^w$ and
  $p^{\theta(w)}$. The isomorphism $\phi_w(u)$ also depends on our
  fixed retract $\pi : q^w \to s^w$.  By Proposition~\ref{prop:unique-injection},
  all such choices are related
  by the natural action of $\Aut_\cP q^w$ (see Definition~\ref{def:Aut-action-on-QG})
  A similar group action
  appears in the identification of $\tilde \Gr_\cP((\sigma_0 \cdot_w 0)
  - u, p^{\theta(w)})$ with $\mathfrak{L}((\sigma_0 \cdot_w
  0)-u,\theta(w))$ (see \cite{Lus00b}).   Via
  the isomorphisms $\phi_w(u)$, the group
  $\Aut_\cP q^w$ acts on the space of constructible functions on
  $\bigsqcup_v \mathfrak{L}(v,w)$ and
  $L_w$ is a subspace of the space of invariant functions.  The pullback
  $(\psi_{\theta(w)} \circ \phi_w^{-1})^*$ acting on the space of
  invariant functions is independent of the choice of $\pi$ and the
  chosen identification of $q^w$ with $p^{\theta(w)}$.
\end{Remark}

\appendix

\section{Isomorphisms of varieties}  \label{Shipman-Appendix}

After an earlier version \cite{ST09} of the current paper was
released, it was proven in \cite{Shi10} that the grassmannian type
varieties $\tilde \Gr_\cP(v,p^w)$ defined by Lusztig are indeed
isomorphic as algebraic varieties to the lagrangian Nakajima quiver
varieties $\mathfrak {L}(v,w)$. A simple ``duality" map gives an
isomorphism of varieties between the quiver grassmannian
${\Gr}_\cP(v,\tilde q^w)$ and $\tilde \Gr_\cP(v,p^w)$. The purpose of this
appendix is to describe this map precisely, and from there to
conclude that the map from $\NGr_\cP(v,q^w) \simeq \Gr_\cP(v,\tilde q^w)$ to $\mathfrak{L}(v,w)$
constructed in Theorem \ref{thm:QG=QV} is in fact an isomorphism of
algebraic varieties. An alternative approach (not pursued here)
would be an injective version of the argument of \cite{Shi10} that
would directly show that $\Gr_\cP(v,\tilde q^w)$ is isomorphic to
$\mathfrak{L}(v,w)$.

Let $i \in Q_0$ and fix a non-degenerate bilinear pairing
\[
  \langle \cdot, \cdot \rangle_{s^i} : s^i \times s^i \to \C,
\]
and a retract $\pi : q^i \to s^i$ of $\cP_0$-modules.  For a path
$\beta = a_1 \cdots a_n$ in the double quiver $\tilde Q$, let
\begin{equation} \label{eq:alg-anti-involution}
  \beta^\vee = \bar a_n \cdots \bar a_1
\end{equation}
be the reverse path.  Extending by linearity, this defines an
algebra anti-involution of $\C \tilde Q$ that induces an algebra
anti-involution of $\cP$. Then define a bilinear pairing
\begin{equation} \label{eq:pairing-i}
  \langle \cdot, \cdot \rangle : {\tilde q}^i \times p^i \to \C, \quad
  \langle v, \beta e_i \rangle = \langle \pi(\beta^\vee v), e_i \rangle_{s^i}.
\end{equation}

For $n \ge 0$, let
\begin{align*}
  p^i_n &= \mathcal{P}_{\ge n} e_i \subseteq p^i, \\
  q^i_n &= \{v \in q^i\ |\ \cP_n \cdot v = 0\} = \{v \in {\tilde
  q}^i\ |\ \cP_n \cdot v = 0 \},
\end{align*}
where the last equality holds since ${\tilde q}^i$ contains all
nilpotent elements of $q^i$ by
Lemma~\ref{lem:loc-nil-contained-in-tqw}.  Note that each $q^i_n$ is
finite-dimensional. We have the obvious inclusions
\[
  q^i_0 \subseteq q^i_1 \subseteq q^i_2 \dots,
\]
and it follows from Lemma~\ref{lem:loc-nil-contained-in-tqw} and
Theorem~\ref{thm:ln-injective-hull} that $\tilde q^i =
\bigcup_{n=0}^\infty q^i_n$.  It is clear from the definitions that
\[
  \langle q^i_n, p^i_{n+1} \rangle = 0, \quad \text{for all } n \ge
  0.
\]
Thus we have the induced bilinear pairing on $q^i_n \times
(p^i/p^i_{n+1})$.

\begin{Lemma} \label{lem:pairing}
The pairing
\[
  \langle \cdot, \cdot \rangle : q^i_n \times (p^i/p^i_{n+1})
  \to \C
\]
is non-degenerate.
\end{Lemma}

\begin{proof}
Since $q^i_n$ is nilpotent of degree $n$ and has socle $s^i$, for
all nonzero $v \in q^i_n$, there exists $\beta \in \cP_{\le n}$ such
that $0 \ne \beta \cdot v \in s^i$.  Then $\langle v, \beta^\vee e_i
\rangle \ne 0$. Thus, it suffices to show that $\dim (p^i/p^i_{n+1})
\le \dim q^i_n$.  Now, $(p^i/p^i_{n+1})^*$ is naturally a right
$\cP$-module. Via the anti-involution
\eqref{eq:alg-anti-involution}, this becomes a nilpotent left
$\cP$-module with socle $s^i$.  Therefore, by
Proposition~\ref{prop:unique-injection}, $(p^i/p^i_{n+1})^*$ injects
into $\tilde q^i$.  It is clear that the image of this injection is
contained in $q^i_n$ and thus the result follows since $q^i_n$ is
finite-dimensional.
\end{proof}

We then have the following corollary, whose proof is immediate.

\begin{Corollary}
The pairing \eqref{eq:pairing-i} is non-degenerate.  Furthermore,
\[
  \tilde q^i \cong \{f \in \Hom_\C (p^i,\C)\ |\ f|_{p^i_n} = 0
  \text{ for } n \gg 0\}
\]
as $\cP$-modules, where the $\cP$-module structure on the right hand
side is given by
\[
  (\beta \cdot f')(v) = f'(\beta^\vee \cdot v), \quad \beta \in \cP,\
  v \in p^i,\ f' \in \{f \in \Hom_\C (p^i,\C)\ |\ f|_{p^i_n} = 0
  \text{ for } n \gg 0\}.
\]
\end{Corollary}

\begin{Remark}
One should compare this result to
Definition~\ref{def:finite-type-injectives} and
Lemma~\ref{lem:injectives} in finite type.
\end{Remark}

Recall that, for $w = \sum_i w_i i \in \N Q_0$, we have
\[
  s^w = \bigoplus_i (s^i)^{\oplus w_i},\quad p^w = \bigoplus_i
  (p^i)^{\oplus w_i},\quad \tilde q^w = \bigoplus_i (\tilde q^i)^{\oplus w_i}.
\]
By declaring distinct summands to be orthogonal, we have a
non-degenerate bilinear pairing
\begin{equation} \label{eq:pairing}
  \langle \cdot, \cdot \rangle : \tilde q^w \times p^w \to \C.
\end{equation}
For a subspace $U$ of $\tilde q^w$, define the subspace
\[
  U^\perp = \{v \in p^w\ |\ \langle v',v \rangle = 0 \text{ for all
  } v' \in U\}
\]
of $p^w$.  Similarly, for a subspace $U$ of $p^w$, define the
subspace $U^\perp$ of $\tilde q^w$.

\begin{Proposition} \label{prop:qg-equality}
For $U \in \Gr_\cP(v,\tilde q^w)$, we have $U^\perp \in \tilde
\Gr_\cP(v,p^w)$, and the map
\[
  \Gr_\cP(v,\tilde q^w) \to \tilde \Gr_\cP(v,p^w),\quad U \mapsto
  U^\perp,
\]
is an isomorphism of algebraic varieties.
\end{Proposition}

\begin{proof}
It follows from the definition of the pairing \eqref{eq:pairing}
that $U$ is a submodule of $\tilde q^w$ if and only if $U^\perp$ is
a submodule of $p^w$.  Also, note that $U \subseteq \tilde q^w$ is
finite-dimensional if and only if $U \subseteq q^w_n$ for some $n$.
Therefore, it follows from Lemma~\ref{lem:pairing} that the maps $U
\mapsto U^\perp$ (in either direction) are mutually inverse
bijections between $\Gr_\cP(v,\tilde q^w)$ and $\tilde
\Gr_\cP(v,p^w)$.  Since these maps are clearly algebraic, the result
follows.
\end{proof}

\begin{Theorem} \label{thm:QG=QV-varieties}
The quiver grassmannian $\Gr_\cP(v,\tilde q^w)$ is isomorphic to the
lagrangian Nakajima quiver variety $\mathfrak{L}(v,w)$ as an
algebraic variety.
\end{Theorem}

\begin{proof}
This follows from the isomorphisms of algebraic varieties
\[
  \Gr_\cP(v,\tilde q^w) \cong \tilde \Gr_\cP(v,p^w)
  \cong \mathfrak{L}(v,w).
\]
The
first isomorphism is Proposition~\ref{prop:qg-equality} and the
second is Proposition~\ref{prop:QGproj=QV}.
\end{proof}

\begin{Corollary} \label{cor:QG=QV-variety-map}
The map $\bar \iota: \Gr_\cP(v,\tilde q^w) \to \mathfrak{L}(v,w)$ of
Theorem~4.4 is an isomorphism of algebraic varieties.
\end{Corollary}

\begin{proof}
By Theorem~\ref{thm:QG=QV-varieties}, we know that $\Gr_\cP(v,\tilde q^w)$
and $\mathfrak{L}(v,w)$ are isomorphic as algebraic varieties. Since
$\bar \iota$ is a bijective algebraic map by
Theorem~\ref{thm:QG=QV}, the result follows by \cite[Lemma~1]{Kal05}
(while the result there is stated for irreducible varieties, the
proof applies to reducible ones -- the only difference is that the
normalization is now a disjoint union of components).
\end{proof}

\section{Erratum} \label{sec:erratum}

The original published version of this paper contained the following errors.  We thank Sarah Scherotzke for bringing this to our attention.

\begin{Error} \label{e1} If $\mathfrak{g}$ is not of finite or affine type, then the Nakajima quiver variety $\Lambda(v,w)$ is not actually isomorphic to the variety $\Gr(v, q^w)$ of all $v$-dimensional subrepresentations of the injective module $q^w$. In fact, beyond affine type, $\Gr(v, q^w)$ does not have a natural variety structure, or, at least, is not finite dimensional. This is because there are continuous families of non-isomorphic modules, all of which have a nontrivial extension with some one-dimensional simple module $S_i$.

There are two ways to modify the statement to make it true, and, with either of these modifications, the work in the original paper does prove the correct result. One must replace $\Gr(v, q^w)$ with either the variety $\NGr(v, q^w)$ of nilpotent $v$-dimensional subrepresentations of $q^w$, or with the variety  $\Gr(v, \tilde q^w)$ of all $v$-dimensional subrepresentations, but where the injective hull $q^w$ in the category of all representations of the preprojective algebra has been replaced with the injective hull $\tilde q^w$ in the category of locally nilpotent representations. Our work shows that these are naturally isomorphic, and are also isomorphic to $\Lambda(v,w)$.
\end{Error}

\begin{Error} Lemma 2.9 (which essentially asserted that $\Gr(v, q^w)$ and $\NGr(v, q^w)$ were isomorphic) is false beyond affine type, and should be removed. The proof is simply incorrect. In fact, this caused most of the issues in \ref{e1}.
\end{Error}

\section*{Acknowledgements}

The authors would like to thank
B.~Leclerc who, after hearing some of the preliminary results of the
current paper, suggested extending these results to graded/cyclic
versions. They are also grateful to W.~Crawley-Boevey for many
helpful discussions and for suggesting the proof of
Proposition~\ref{prop:injective-ln-condition}. Furthermore, they
would like to thank P.~Etingof, A. Hubery, H.~Nakajima, M.~Roth, O.~Schiffmann, and
I.~Shipman for useful conversations and S.-J.~Kang,
Y.-T.~Oh, and the Korean Mathematical Society for the invitation to
participate in the 2008 Global KMS International Conference in Jeju,
Korea where the ideas in the current paper were originally
developed.


\bibliographystyle{abbrv}
\bibliography{biblist}

\bigskip

\end{document}